\documentclass[11pt]{amsart}

\usepackage{latexsym}
\usepackage{epic}
\usepackage[leqno]{amsmath}
\usepackage{amssymb}
\usepackage{graphicx}
\usepackage{color}

\newtheorem{lemma}{Lemma}[section]
\newtheorem{theorem}[lemma]{Theorem}
\newtheorem{proposition}[lemma]{Proposition}
\newtheorem{corollary}[lemma]{Corollary}

\newcommand{\Tt}{\mbox{$\mathcal T$}}

\newcommand{\R}{{\mathbb R}}

\newcommand{\Tc}{\mathcal{T}}
\newcommand{\mcup}{\biguplus}

\newcommand{\upi}{\underline{\Pi}}
\newcommand{\usigma}{\underline{\Sigma}}

\newcommand{\setpa}{\mathbb{P}}

\begin{document}

\title{Representing Partitions on Trees}

\author{K. T. Huber}  
\address{
School of Computing Sciences, University of East Anglia,
Norwich, United Kingdom}
\email{katharina.huber@cmp.uea.ac.uk}
\thanks{KTH and VM would like to thank the Biomathematics 
Research Centre, Department of Mathematics and Statistics, 
University of Canterbury, Christchurch, New Zealand, and the
Department of Mathematics, National University of Singapore, Singapore 
for hosting them during part of the work. CS was supported by the New 
Zealand Marsden Fund and The Allan Wilson Centre for 
Molecular Ecology and Evolution. TW would like to acknowledge support from
Singapore MOE (grant$\#$: R-146-000-134-112)}

\author{V. Moulton}  
\address{
School of Computing Sciences, University of East Anglia,
Norwich, United Kingdom}
\email{vincent.moulton@cmp.uea.ac.uk}

\author{C. Semple}  
\address{
Biomathematics Research Centre, Department of 
Mathematics and Statistics, University of Canterbury, 
Christchurch, New Zealand}
\email{charles.semple@canterbury.ac.nz}

 \author{T. Wu}  
\address{School of Computing Sciences, University of East Anglia,
Norwich, United Kingdom
}
\email{taoyang.wu@gmail.com}

\date{\today}

\maketitle

\begin{abstract}
In evolutionary biology, biologists often face the
problem of constructing a phylogenetic tree on a set $X$ of species from 
a multiset $\Pi$ of partitions corresponding to various attributes of these species.  One approach that is used to 
solve this problem is to try instead to
associate a tree (or even a network) to the multiset $\Sigma_{\Pi}$
consisting of all those bipartitions $\{A,X-A\}$
with $A$ a part of some partition in $\Pi$.
The rational behind this approach is that 
a phylogenetic tree with leaf set $X$
can be uniquely represented by the set of 
bipartitions of $X$ induced by its edges.
Motivated by these considerations, given a multiset 
$\Sigma$ of bipartitions corresponding to 
a phylogenetic tree on $X$, in this paper
we introduce and study the set $\setpa(\Sigma)$ 
consisting of those multisets of partitions $\Pi$ 
of $X$ with $\Sigma_{\Pi}=\Sigma$.
More specifically, we characterize when
$\setpa(\Sigma)$ is non-empty, and also 
identify some partitions in $\setpa(\Sigma)$
that are of maximum and minimum size.
We also show that it is NP-complete
to decide when $\setpa(\Sigma)$ is non-empty 
in case $\Sigma$ is an arbitrary multiset of 
bipartitions of $X$. Ultimately, we hope that by gaining a 
better understanding of the mapping that 
takes an arbitrary partition system $\Pi$  
to the multiset $\Sigma_{\Pi}$,  
we will obtain new insights
into the use of median networks and, more generally, 
split-networks to visualize sets of partitions.\\




\end{abstract}

\section{Introduction}

In evolutionary biology, biologists are often faced with the task of 
constructing a phylogenetic tree (i.e. an unrooted, 
edge-weighted tree without degree-two vertices and leaf set $X$) 
that represents a multiset $\Pi$ of partitions of a finite 
set $X$ of species or taxa.
Such multisets of partitions (or {\em partition systems}) usually arise 
from some collection of attributes or states of the 
species in question (e.g. ``wings'' versus ``no wings'' or
the four possible nucleotides in the columns of some 
molecular sequence alignment).
It is well-known that a phylogenetic tree
with leaf set $X$  is determined by the bipartitions or {\em splits} of $X$ that 
are induced by its edges \cite{B71}.
Hence, when trying to derive such trees from multi-state data, biologists sometimes
consider instead the multiset $\Sigma_{\Pi}$
of splits of $X$ consisting of all those $\{A,X-A\}$
with  $A \in \pi$ for some partition $\pi$ contained in a partition system $\Pi$ induced by the data~\cite{AB08,HMS04,DRS10}. 
The aim then becomes associating a tree (or possibly a 
network) to the multiset $\Sigma_{\Pi}$.

As an example of this process, for the
set $X = \{1,2,3,4,5,6\}$, consider the set of partitions 
$\Pi_1 = \{123|4|56,1|2|3456,3|12456,5|6|1234\}$ on $X$
(where, e.g., $123|4|56$ denotes the 
partition $\{\{1,2,3\},\{4\},\{5,6\}\}$).
Then the multiset $\Sigma_{\Pi_1}$ is 
represented (uniquely) by the phylogenetic tree
in Fig.~\ref{trees}.
Intriguingly,  $\Pi_1$ is not the 
only partition system that gives rise to the
tree depicted in Fig.~\ref{trees}.
For example, the set 
$\Pi_2 = \{123|4|5|6,1|2|3456,3|12456,56|1234\}$
gives rise to precisely the same tree
(or, in other words, $\Sigma_{\Pi_1} = \Sigma_{\Pi_2}$).
Thus, given a multiset $\Sigma$ of splits of $X$
that is {\em compatible} (i.e. corresponds to a phylogenetic tree),
it is of interest to better understand the set
$\setpa(\Sigma)$ that consists of all those
partition systems $\Pi$ on $X$ such that 
$\Sigma_{\Pi}=\Sigma$ holds.
As we shall see, the set $\setpa(\Sigma)$
can be quite complicated
in general. For example, even for the simple tree
in Fig.~\ref{trees} it can be shown that   
$\setpa(\Sigma_{\Pi_1})$ consists of 
$\Pi_1$, $\Pi_2$ as well as the sets
$\Pi_3 = \{12|3|4|56,1|2|3|456,5|6|1234\}$,
$\Pi_4 = \{12|3|4|5|6,1|2|3|456,56|1234\}$,
$\Pi_5 = \{1|2|3|4|56,12|3|456,5|6|1234\}$, and
$\Pi_6 = \{1|2|3|4|5|6,12|3|456,56|1234\}$.

\begin{figure}[h]
\center
\includegraphics[scale=0.35]{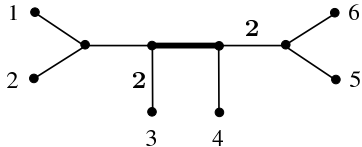}          
\caption{A tree that represents the multiset 
 $\Sigma_{\Pi_1}$ of splits on the set $X =\{1,2,3,4,5,6\}$
given in the text.
The removal of any edge of the tree gives
a split of $X$, with multiplicity 
given by the weight in bold assigned to the edge 
(all unlabelled edges have weight 1). For example, 
the bold edge gives rise to the split $123|456$.}
\label{trees}
\end{figure}

Although these considerations all appear rather 
abstract, our study of the set $\setpa(\Sigma)$ was
motivated by its appearance in the 
construction of {\em median networks}.
These networks generalize phylogenetic trees
and are commonly used to visualize complex 
evolutionary relationships arising from mitochodrial 
sequences \cite{BFR99,BYB09,Mor10}. Median networks 
can be directly constructed from splits \cite{DHHM97}.
Moreover, given a multiple sequence alignment
of a set $X$ of sequences, one way that is used to derive 
splits before constructing a median network is to
convert each non-constant column into a partition of $X$
so as to give a multiset $\Pi$ of partitions of $X$, and
then construct the multiset $\Sigma_{\Pi}$ (see, e.g.~\cite{AB08,DRS10}).
Thus, we expect that by gaining a better understanding
of the set $\setpa(\Sigma)$ (also for general split 
systems $\Sigma$) we will be able to obtain new insights
into the use of median networks (and more generally
{\em split-networks}; cf. \cite{HS10}) to represent partitions.
In addition, through considerations such as those presented 
in \cite{MR1907821}, we hope that
our results will help to further clarify the relationships 
between median and {\em quasi-median networks} given in \cite{HMS04}.

We now present an overview of our main results.
In the following two sections we present some
notation and terminology 
as well as some preliminary results that will be used 
throughout the paper. Then, in Section~\ref{eventrees}, 
we characterize those compatible multisets of
splits $\Sigma$ for which $\setpa(\Sigma)$ is non-empty
(Theorem~\ref{main1}). In addition, for $\Sigma$
a compatible multiset of splits of $X$, we show 
that if $\setpa(\Sigma)$ is non-empty then there is
always a unique partition system $\Pi$ in $\setpa(\Sigma)$
which is {\em strongly compatible}, i.e. 
for all $\pi_1,\pi_2 \in \Pi$ either $\pi_1=\pi_2$ or there 
is some $A \in \pi_1, B\in \pi_2$ such that $A \cup B =X$ \cite{DMS97}.
For example, for the multiset  $\Sigma$ of splits
giving rise to the tree depicted in Fig.~\ref{trees},
the set $\Pi_1$ is the unique strongly compatible
partition system in $\setpa(\Sigma)$.

As the example above illustrates, 
the size of the elements in $\setpa(\Sigma)$
can vary (e.g. the size of $\Pi_1$ is 4
whereas $\Pi_6$ has size 3). We are therefore interested
in understanding the maximum- and minimum-sized 
elements in this set.
In Section~\ref{maxsized}, we show that 
the unique, strongly compatible partition system 
in $\setpa(\Sigma)$ is always of maximum size.
In the subsequent section, we then focus
on minimum-sized elements of $\setpa(\Sigma)$, 
giving a method to construct 
such a partition system.
In general, it appears to be a difficult problem
to characterize the maximum-sized and minimum-sized
partition systems in $\setpa(\Sigma)$ for a
compatible multiset $\Sigma$ of splits. However, in 
Section~\ref{minsized} we characterize the 
minimum-sized elements for a special type
of multiset of splits that corresponds to 
a rooted tree in which the root has the 
same distance in the tree to all of the leaves.

In Section~\ref{hard}, we investigate a related algorithmic question: 
Given an arbitrary split system $\Sigma$ on $X$, 
can we decide in polynomial time in the size 
of $X$ if there exists a partition system
$\Pi$ of $X$ such that $\Sigma_{\Pi}=\Sigma$? 
By reduction from the Cubic Edge Colouring problem, 
we show that this problem is NP-complete, even if $\Sigma$ is an arbitrary {\em set}, that is, the 
multiplicity of each split in $\Sigma$ is equal to one (Theorem~\ref{NP-complete}). This
indicates that it might be difficult in general
to extend our main results to arbitrary multisets 
of splits. In the final section, 
we discuss how the mapping from partition
systems to split systems  given by taking a partition system $\Pi$ to 
the split system $\Sigma_{\Pi}$ could be studied in a more general 
setting, and mention some open problems that this leads to. 

Before proceeding we note that 
the problem of representing partitions (or characters)
by trees has also been studied in the context of the
{\em perfect phylogeny problem}.
This problem is concerned with 
representing partitions {\em convexly}
on a phylogenetic tree, and a great
deal of related theory has been developed 
(cf. e.\,g.\,\cite[Chapter 4]{SS03} and 
e.\,g.\,\cite{GH06,GLG11,SG10} for more recent results).  
However, this approach differs from ours since, for example,
there exist sets $\Pi$ of partitions all of 
whose elements are convex on some phylogenetic tree for which 
$\Sigma_{\Pi}$ is not compatible.

\section{Preliminaries}
\label{notation}

\noindent {\bf Multisets.} If $S$ is a finite non-empty set, a 
{\em multiset} chosen from $S$ 
is a function $m$ from $S$ into the set of non-negative integers 
${\mathbb Z}^{\ge 0}$. The set $S$ is sometimes called its {\em
underlying set}. For an element $t$ in $S$, the value 
$m(t)$ is the {\em multiplicity} of $t$. For example, 
let $S=\{1,2,3,4\}$. Then the multiset $\{1, 1, 2, 2, 2, 3\}$ denotes 
the function $m$ from $S$ into ${\mathbb Z}^{\ge 0}$ with $m(1)=2$, 
$m(2)=3$, $m(3)=1$, and $m(4)=0$. The multiplicity of $2$ is $3$, while 
the multiplicity of $4$ is $0$. The size $|M|$ of a multiset $M$ 
with underlying set $S$ is the sum of the multiplicities over all elements in
$S$. Let $m_1$ and $m_2$ be two functions from 
$S$ into ${\mathbb Z}^{\ge 0}$, and let $S_1$ and $S_2$ denote the multisets 
corresponding to $m_1$ and $m_2$, respectively. We denote the {\em multiset 
union} of $S_1$ and $S_2$ by $S_1\mcup S_2$, where $S_1\mcup S_2$ is the 
function from $S$ into ${\mathbb Z}^{\ge 0}$ defined by $m_1(t)+m_2(t)$ 
for all $t\in S$. Moreover, we denote the {\em multiset difference}
of $S_1$ and $S_2$ by $S_1- S_2$ where $S_1- S_2$ is the 
function from $S$ into ${\mathbb Z}^{\ge 0}$ defined by 
$\max\{0, m_1(t)-m_2(t)\}$ for all $t\in S$.  

\noindent {\bf Weak $X$-trees.} Throughout the paper, $X$ will always denote a  finite set of size at least two. 
A {\em weak $X$-tree} $\mathcal T$ is an 
ordered pair 
$(T; \phi)$, where $T$ is a tree with vertex set $V$ and $\phi:X\rightarrow V$ 
is a map with the property that, for each vertex $v\in V$ of degree one, 
$v\in \phi(X)$. For convenience, we 
refer to the vertices and edges 
of $T$ as the vertices and edges of $\mathcal T$, respectively, and write 
$V({\mathcal T})$ for $V(T)$ and $E({\mathcal T})$ for $E(T)$. A vertex $v$ 
of $\mathcal T$ is {\em labelled} if $v\in \phi(X)$; otherwise, $v$ is 
{\em unlabelled}. Given $u,v\in V$, we denote the length of the path joining 
$u$ and $v$ by $d_T(u,v)$. Sometimes we will
also use $d_{\mathcal T}(u,v)$ rather than $d_T(u,v)$.
A weak $X$-tree $\mathcal T$ is an {\em $X$-tree} if it additionally has the 
property that each degree-two vertex is labelled. Note that a 
phylogenetic $X$-tree $\mathcal T$ is an $X$-tree in which $\phi$ is 
a bijective map from $X$ to the leaf set of $\mathcal T$. We say that two
weak $X$-trees $\mathcal T=(T; \phi)$
and $\mathcal T'=(T'; \phi')$ are {\em isomorphic}, denoted by
$\mathcal T\cong \mathcal T'$, if
there exists a bijective map $\psi:V(T)\to V(T')$ 
that induces a graph isomorphism between $T$ and $T'$
for which $\phi'(x)=\psi(\phi(x))$ holds for all $x\in X$.

Note that weak $X$-trees are closely related to 
weighted $X$-trees, where an $X$-tree is {\em weighted} 
if each edge is assigned a positive integer weight. 
For example, the phylogenetic tree depicted in Fig.~\ref{trees} 
is equivalent to a weak $X$-tree in which each edge with 
weight 2 is subdivided into two edges by inserting an 
extra vertex. Indeed, we can translate between 
weighted $X$-trees and weak $X$-trees in general 
by inserting or suppressing unlabelled degree 2 vertices 
in a similar manner. However, in this paper we will
use weak $X$-trees rather than weighted $X$-trees
since they are more convenient for many of
our proofs (e.g. their vertices and edges can 
be used to represent certain partition systems).

\noindent{\bf Compatible split systems and hierarchies.} 
As mentioned in the introduction, a split of $X$ or, 
equivalently, an {\em $X$-split} is a bipartition of $X$ into two non-empty 
sets, that is, a partition $\pi=\{A_1, A_2, \ldots, A_t\}$ 
of $X$ with $t\geq 2$ in which each subset $A_i$, $i\in \{1,\ldots, t\}$,  
is non-empty and $t=2$ (rather than
$t\geq2 $ as is the case for a general partition of $X$). We will refer
to the subsets $A_i$ as {\em parts} of $\pi$ and, to simplify notation, 
we write $\{A_1, A_2, \ldots, A_t\}$ as $A_1|A_2|\cdots|A_t$, where 
the ordering of the parts of $\pi$ is irrelevant. 
A multiset of 
$X$-splits is called a {\em split system on $X$}. Split systems on $X$ 
naturally arise in the context of weak $X$-trees. In particular, let 
${\mathcal T}=(T; \phi)$ be a weak $X$-tree and let $e$ be an edge of 
$\mathcal T$. We denote by $\sigma_e$ the $X$-split $A|(X-A)$, where $A$ 
is one of the two 
maximal subsets of $X$ such that $e$ is not traversed on the path 
from $\phi(x)$ to $\phi(y)$ for all 
$x,y\in A$. This $X$-split {\em corresponds} to, or 
equivalently is {\em displayed} by, $e$ in $\mathcal T$. Note that, as 
$\mathcal T$ is a weak $X$-tree, it is possible that, for distinct edges 
$e$ and $f$, we have $\sigma_e=\sigma_f$. We denote the split system on $X$ 
corresponding to the edges of $\mathcal T$ by $\Sigma({\mathcal T})$, that is,
$$\Sigma({\mathcal T})=\mcup_{e\in E({\mathcal T})} \{\sigma_e\}.$$


A pair of $X$-splits $A_1|B_1$ and $A_2|B_2$ is {\em compatible} if at 
least one of the sets $A_1\cap A_2$, $A_1\cap B_2$, $B_1\cap A_2$, and 
$B_1\cap B_2$ is the empty set. A split system $\Sigma$ on $X$ is 
{\em compatible} 
if the splits in $\Sigma$ are pairwise compatible. The following theorem is 
a straightforward generalization of the Splits-Equivalence Theorem~\cite{B71} 
(also see~\cite[Theorem 3.1.4]{SS03}).

\begin{theorem}\label{bun-thm}
Let $\Sigma$ be a split system on $X$. Then there is a weak $X$-tree 
$\mathcal T$ such that $\Sigma=\Sigma({\mathcal T})$ if and only if 
$\Sigma$ is compatible. Moreover, if such a weak $X$-tree exists, then, 
up to isomorphism, $\mathcal T$ is unique.
\label{buneman}
\end{theorem}

In light of this last result, if $\Sigma$ is a compatible split system 
on $X$, we denote the unique weak $X$-tree $\mathcal T$ for which 
$\Sigma({\mathcal T})=\Sigma$ holds by ${\mathcal T}_{\Sigma}$.
Note that in case $\Sigma=\Sigma(\Tt)$ for a weak $X$-tree $\Tt$, we
will write $\setpa(\Tt)$ rather than $\setpa(\Sigma(\Tt))$.

An analogue of Theorem~\ref{bun-thm} holds for 
trees having a root.
To make this statement more precise, we introduce some 
further terminology. A {\em rooted weak X-tree} 
$\mathcal T_{\rho}$ is an ordered pair $(T_{\rho}; \phi)$, where $T_{\rho}$ 
is a rooted tree with root $\rho$ which has degree at least two and vertex set 
$V$, and $\phi: X\to V-\{\rho\}$ is a map with the 
property that, for each vertex 
$v\in V$ of degree one, $v\in \phi(X)$. Note that if we view 
$\mathcal T_{\rho}$ as an unrooted tree with $\rho$ as ordinary 
interior vertex, we obtain a weak $X$-tree. We denote this weak 
$X$-tree by $\mathcal T^-_{\rho}$.

A {\em cluster} of $X$ is a non-empty subset of $X$ and it is {\em proper} 
if it is distinct from $X$. Let $\mathcal T_{\rho}$ be a rooted weak 
$X$-tree and let $e$ be an edge of $\mathcal T_{\rho}$. The proper subset 
of $X$ consisting of those elements that label a vertex in 
$\mathcal T_{\rho}$ whose path to the root traverses $e$ is denoted 
by $C_e$. This cluster $C_e$ {\em corresponds} to, or equivalently is 
{\em displayed} by, $e$ in $\mathcal T_{\rho}$. We denote the multiset 
of clusters of $X$ corresponding to the edges of $\mathcal T_{\rho}$ 
by $\mathcal H(\mathcal T_{\rho})$, that is,
$$\mathcal H(\mathcal T_{\rho})=\mcup_{e\in E(\mathcal T_{\rho})} \{C_e\}.$$
It is straightforward to show that this multiset 
of subsets of $X$
is a {\em hierarchy}, that is, for all 
$A, B\in \mathcal H(\mathcal T_{\rho})$, we 
have  $A\cap B\in \{\emptyset, A, B\}$. 
The next result is the aforementioned analogue of 
Theorem~\ref{bun-thm}. We omit the routine proof.

\begin{theorem}
Let $\mathcal H$ be a multiset of proper clusters of $X$ 
whose union is $X$. Then there is a rooted weak $X$-tree 
$\mathcal T_{\rho}$ such that $\mathcal H=\mathcal H(\mathcal T_{\rho})$ 
if and only if $\mathcal H$ is a hierarchy on $X$. Moreover, if such a 
rooted weak $X$-tree exists, then, up to isomorphism, 
$\mathcal T_{\rho}$ is unique.
\label{hbuneman}
\end{theorem}

\noindent {\bf Partition systems.} 
A partition system $\Pi$ of $X$ is {\em compatible} if 
$\Sigma_{\Pi}$ is compatible. Again following Theorem~\ref{buneman}, 
if $\Pi$ is a compatible partition system on $X$, we denote the weak $X$-tree 
$\mathcal T$ for which 
$\Sigma(\Tt)=\Sigma_{\Pi}$ holds by ${\mathcal T}_{\Pi}$. 
Similarly, a partition system  $\Pi$ on $X$ is {\em hierarchical} if
the set $\bigcup_{\pi\in \Pi} \pi$ of all 
subsets of $X$ that appear as a part in some partition in $\Pi$
is a hierarchy. Observe that, if $\Pi$ 
is hierarchical, then every subset of $\Pi$ is a hierarchical partition 
system on $X$. Furthermore, if $\pi_1,\pi_2\in \Pi$ and $\Pi$ is 
hierarchical, then, for each $A\in \pi_1$, either $A$ is a subset 
of a part in $\pi_2$ or $A$ is the disjoint union of parts in $\pi_2$.

Now, given a partition $\pi$ of $X$, we let 
$\Sigma_{\pi}=\mcup_{A\in \pi} \{A|(X-A)\}$, i.e. 
the multiset of bipartitions $A|(X-A)$ with $A \in \pi$.
The proof of the next result 
follows immediately from the respective definitions. 

\begin{lemma}\label{lem:union}
Let $\pi$ be a partition of $X$, $\Sigma$ be a split system on $X$, 
and $\Pi\in \setpa(\Sigma)$. Then the
following hold.
\begin{enumerate}
\item[(i)] $\Sigma_{\{\pi\}\uplus\Pi}=\Sigma\uplus\Sigma_{\pi}$.
\item[(ii)] If $\pi\in\Pi$, then $\Sigma_{\Pi-\{\pi\}}=\Sigma-\Sigma_{\pi}$.
\end{enumerate}
\end{lemma}

\section{Displaying Partition Systems}
\label{display}

In this section, we describe how weak $X$-trees 
can be used to represent partition systems. 
Let ${\mathcal T}=(T; \phi)$ be a weak $X$-tree and let $\pi$ be 
a partition of $X$. A subset $E_{\pi}\subseteq E(\mathcal T)$ 
of edges of $\mathcal T$ {\em displays} $\pi$ if there is a 
bijection $\xi_{\pi}:\pi\to E_{\pi}$ such that, for each $A\in \pi$, 
the $X$-split corresponding to the edge $\xi_{\pi}(A)$ is $A|(X-A)$. 
For convenience, 
if there exists such a subset $E_{\pi}$ of edges of $\mathcal T$, then we 
say that $\mathcal T$ {\em displays} $\pi$. Note that such a subset $E_{\pi}$ 
need not be unique.

For a compatible partition system $\Pi$ on $X$, the 
following two lemmas
that we use later on describe 
how ${\mathcal T}_{\Pi}$ displays the partitions in $\Pi$. 
Suppose $\Tt=(T; \phi)$ is a weak $X$-tree and let $e$ denote an edge
of $T$. Then we denote by $\Tt\backslash e$ the set of components of
$\Tt$ obtained by deleting $e$ from $T$. More generally, for $E$ a 
non-empty subset of edges of $T$ we denote by $\Tt\backslash E$ the 
set of components of $\Tt$ obtained by deleting all edges in $E$ from
$T$.

\begin{lemma}
Let $\Pi$ be a compatible partition system on $X$ and let $u$ be a vertex of 
$\Tt_{\Pi}=(T; \phi)$ such that $\phi^{-1}(u)\not=\emptyset$.
Let $e$ be an edge of $\Tt_{\Pi}$ incident with $u$ and let 
$B\in \sigma_e$ such
that $\phi^{-1}(u)\subseteq B$. Then there exists some $\pi\in \Pi$
such that $B\in \pi$.
\label{labelvertex}
\end{lemma}

\begin{proof}
Let $A|B$ be the $X$-split corresponding to $e$, where 
$\phi^{-1}(u)\subseteq B$. Since $\Sigma({\mathcal T}_{\Pi})=\Sigma_{\Pi}$, 
it follows that there is a partition $\pi$ in $\Pi$ such that either 
$A\in \pi$ or $B\in \pi$. Suppose that $A\in \pi$, but $B\not\in \pi$. 
Then $|\pi|\geq 3$ and there exists a part $D\in \pi$ such that 
$\phi^{-1}(u)\subseteq D$. Hence, $A\cap D=\emptyset$ and $A\cup D\not=X$.
Since $\pi$ is displayed by $\Tt_{\Pi}$ and $D\in \pi$ there exists some
edge $e'$ of $\Tt_{\Pi}$ such that $\sigma_{e'}=D|X-D$. But then
either $e'$ is an edge in the connected component $Z$
of $\Tt_{\Pi}\backslash e$ that contains $u$ or $e'$ is contained in 
the other component $Z'$ of $\Tt_{\Pi}\backslash e$. In the former
case it follows that $A\subsetneq D$ and so $A\cap D\not=\emptyset$
which is impossible. Thus, $e'$ is an edge in $Z'$. But then
the connected component of $\Tt_{\Pi}\backslash e'$ 
that contains $u$ must contain
$Z$. Since $\phi^{-1}(u)\subseteq B\cap D$ it follows that
$B\subseteq D$ and thus $X=A\cup B\subseteq A\cup D\not=X$ which is
also impossible.
Thus $B\in \pi$.
\end{proof}


\begin{lemma}
Let $\Pi$ be a compatible partition system on $X$ and let $\pi$ 
be an element of $\Pi$. Let $E_{\pi}$ be a subset of edges of 
$\Tt_{\Pi}=(T; \phi)$ that displays $\pi$. Then the following holds:
\begin{itemize}
\item[{\rm (i)}] Denoting by  $V_1, V_2, \ldots, V_k$, $k\geq 1$,  
the vertex sets of 
the components of $\Tt_{\Pi}\backslash E_{\pi}$, then  $k=|\pi|+1$ and 
$$\{\phi^{-1}(V_1), \phi^{-1}(V_2), \ldots, \phi^{-1}(V_k)\}
=\pi\cup \{\emptyset\}.$$

\item[{\rm (ii)}] For every pair of labelled vertices $u$ and $v$ 
of $\Tt_{\Pi}$, the path joining $u$ to $v$ contains exactly $0$ 
or $2$ edges of $E_{\pi}$.
\end{itemize}
\label{displaypi}
\end{lemma}

\begin{proof}
 We first assume that $|\pi|=2$, that is, $\pi=A|B$ for some split $A|B$ of $X$. 
Let $E_{\pi}$ be a subset of edges of 
$\Tt_{\Pi}$ that displays $\pi$. Then $E_\pi$ consists of 
two distinct edges $e_1=\{u'_1,u_1\}$ and $e_2=\{u'_2,u_2\}$ 
such that $\sigma_{e_1}=A|B=\sigma_{e_2}$. 
By swapping $u'_1$ and $u_1$, and $u'_2$ and $u_2$ if necessary, 
we may assume that $u'_1$ and $u'_2$ are not contained 
in the shortest path $P$ between $u_1$ and $u_2$.
Moreover, since $\sigma_{e_1}=\sigma_{e_2}$, each vertex of $P$, including $u_1$ and $u_2$, is 
unlabelled and has degree two. Hence the lemma holds for this case.

Next assume that $|\pi|\ge 3$.
Suppose $e\in E_{\pi}$ and $B\in \sigma_e$ with $B\in \pi$. Let $v_B$
denote the end-vertex of $e$ that is contained in the connected
component of $\Tt_{\Pi}\backslash e$ which contains some (and thus
all) $u\in V(T)$ such that $\phi^{-1}(u)\subseteq B$.
 Since $\pi$ is a partition of $X$, 
there is no edge in $E_{\pi}$ on the path from $v_B$ to a vertex $w$ such that 
$\phi^{-1}(w)\subseteq B$. As $\pi$ is a partition of $X$, 
part (i) of the lemma now follows.

For the proof of (ii), let $u$ and $v$ be distinct labelled vertices of 
$\Tt_{\Pi}$. Suppose that the path from $u$ to $v$ contains (in order) 
three edges $e_1$, $e_2$, and $e_3$ of $E_{\pi}$. Since $|\pi|\ge 3$, the 
splits $A_1|B_1$, $A_2|B_2$, and $A_3|B_3$ 
corresponding to $e_1$, $e_2$, and $e_3$, respectively, are distinct. 
Without loss of generality, we may assume that $A_1\subset A_2\subset A_3$. 
But then $B_1\cap A_2$ is non-empty and $B_2\cap A_3$ is non-empty. Since, 
for each $i\in \{1,2,3\}$, at least one of $A_i$ and $B_i$ must be 
contained in $\pi$, it follows that $\pi$ is not a partition of $X$; a
 contradiction. Thus the path from $u$ to $v$ contains at most two edges 
of $E_{\pi}$.

Now suppose that the path from $u$ to $v$ contains exactly one edge 
$e_1$ of $E_{\pi}$. Let $A_1|B_1$ be the split corresponding to $e_1$. 
Without loss of generality, we may assume that 
$\phi^{-1}(u)\subseteq A_1$ and $A_1\in \pi$. 
Then $\phi^{-1}(v)\cap A_1=\emptyset$.
Now $|E_{\pi}|\ge 3$ and no edge in $E_{\pi}-\{e_1\}$ is on the path 
from $u$ to $v$. It follows that, for all edges $e'\in E_{\pi}-\{e_1\}$, 
the component of 
$\Tt_{\Pi}\backslash e'$ that contains $v$ also contains $u$. In particular, 
there is a part in $\pi$ that contains $\phi^{-1}(u)\cup \phi^{-1}(v)$; a 
contradiction as $\phi^{-1}(u)\subseteq A_1$ and 
$\phi^{-1}(v)\cap A_1=\emptyset$. This completes the proof of (ii), and 
thus the proof of the lemma.
\end{proof}

For a tree $T$, the {\em diameter} of $T$, denoted $\Delta(T)$, is
$$\Delta(T)=\max\{d_T(u,v): \mbox{$u$ and $v$ are leaves of $T$}\}.$$
The following corollary is an immediate consequence of 
Lemma~\ref{displaypi}(ii), and gives a lower bound on
the size of a partition in $\setpa(\Sigma)$ for 
$\Sigma$ compatible in terms of the tree corresponding to $\Sigma$.

\begin{corollary}
Let $\Sigma$ be a compatible split system on $X$ and 
let $\Pi\in\setpa(\Sigma)$. Then
\begin{align*}
\Delta({\mathcal T}_{\Sigma})\le 2|\Pi|.
\end{align*}
\label{lem:diam:bound}
\end{corollary}

\section{A Characterization of Compatibility}
\label{eventrees}

In this section, for a given
split system $\Sigma$ on $X$, we characterize when there exists 
a partition system $\Pi$ on $X$ such that $\Sigma_{\Pi}=\Sigma$
(i.e. when $\setpa(\Sigma)$ is non-empty). 
We begin by presenting some definitions.

A {\em 2-colouring} of a graph $G$ is a bipartition of the vertex set of $G$ 
such that no two vertices in a part are joined by an edge.
An {\em even $X$-tree} $\Tt=(T; \phi)$ is a weak $X$-tree with the 
additional property that $d_{\mathcal T}(\phi(x), \phi(y))$ is even 
for all $x, y\in X$. Let $v$ be a vertex of an even $X$-tree $\Tt=(T, \phi)$. 
Then $v$ is {\em even} if there is a leaf $l$ in $\Tt$ such that 
$d_{\mathcal T}(v, l)$ is even; otherwise, $v$ is {\em odd}. Note that 
all leaves of $\Tt$ are even and that we treat the
number zero as an even number. We denote by  
$V_{\mbox{\scriptsize \rm even}}(\Tt)$ the subset of
even vertices of $\Tt$ and by $V_{\mbox{\scriptsize odd}}(\Tt)$
the subset of odd vertices of $\Tt$.

\begin{lemma}
Let $\Tt$ be an even $X$-tree. Then
\begin{itemize}
\item[{\rm (i)}] all labelled vertices of $\Tt$ are even, and

\item[{\rm (ii)}] the even and odd vertices of $\Tt$ induce a 
$2$-colouring of $\Tt$.
\end{itemize}
\label{littleeven}
\end{lemma}

\begin{proof}
Part (i) follows immediately from the definition of an even 
$X$-tree. For part (ii), it is easily checked that every edge is 
incident with exactly one even vertex and one odd vertex, and so 
the even and odd vertices induce a $2$-colouring of $\Tt$.
\end{proof}


Let ${\mathcal T}=(T; \phi)$ be a weak $X$-tree and let $v$ be an unlabelled 
vertex of $\mathcal T$. Then the partition of $X$ {\em displayed} 
by $v$ is precisely the partition $\pi$ in which two elements $x,y\in X$ 
are in the same part of $\pi$ if and only if the path 
from $\phi(x)$ to $\phi(y)$ 
does not pass through $v$. We denote this partition by $\pi(v)$. Note 
that the degree of $v$ equals $|\pi(v)|$. Moreover 
for a graph $G$ and a vertex $v\in V(G)$, 
we denote by $G\backslash v$ the graph obtained from $G$ by deleting $v$ 
and all its incident edges.

\begin{theorem}\label{main1}
Let $\Sigma$ be a compatible split system on $X$. Then the 
following statements are equivalent:
\begin{itemize}
\item[{\rm (i)}] $\Tt_{\Sigma}$ is even.

\item[{\rm (ii)}] There exists a partition system $\Pi$ on $X$ such 
that $\Sigma_{\Pi}=\Sigma$.

\item[{\rm (iii)}] There exists a strongly compatible partition 
system $\Pi_s$ on $X$ such that $\Sigma_{\Pi_s}=\Sigma$.
\end{itemize}
Furthermore, if (iii) holds, then
$$
\Pi_s=\{\pi(v): v\in V_{\mbox{\scriptsize \rm odd}}(\Tt_{\Sigma})\}
$$
is the unique strongly compatible partition system with 
$\Sigma_{\Pi_s}=\Sigma$.
\label{thm:tree:sc}
\end{theorem}

\begin{proof}
Evidently, (iii) implies (ii). To see that (ii) implies (i), 
suppose that $\Pi$ is a partition system on $X$ such that 
$\Sigma_{\Pi}=\Sigma$. Let $\Tt_{\Sigma}=(T, \phi)$ and let $x, y\in X$. 
Then $d_{\mathcal T_{\Sigma}}(\phi(x), \phi(y))$ is equal to the number of 
splits $S$ in $\Sigma_{\Pi}$ for which $x$ and $y$ are in different parts 
of $S$. By Lemma~\ref{displaypi},
each partition in $\Pi$ contributes either $0$ or $2$ such splits. 
Thus $d_{\mathcal T_{\Sigma}}(\phi(x), \phi(y))$ is even
 and, hence, $\Tt_{\Sigma}$ is even.

We next show that (i) implies (iii). 
Suppose that $\Tt_{\Sigma}=(T, \phi)$ is even and put
$V_{\mbox{\scriptsize \rm odd}}=V_{\mbox{\scriptsize \rm odd}}(\Tt_{\Sigma})$. 
By Lemma~\ref{littleeven}(i), $V_{\mbox{\scriptsize \rm odd}}$ 
contains no labelled vertex of $\Tt_{\Sigma}$. Let
$$
\Pi_s=\{\pi(v): v\in V_{\mbox{\scriptsize \rm odd}}\}.
$$
By Lemma~\ref{littleeven}(ii), every edge of $\Tt_{\Sigma}$ is 
incident with exactly 
one vertex in $V_{\mbox{\scriptsize \rm odd}}$ and so it follows that 
$\Sigma_{\Pi_s}=\Sigma$. Furthermore, let 
$v_1$ and $v_2$ be distinct vertices in $V_{\mbox{\scriptsize \rm odd}}$. 
Let $V_2$ (resp.\ $V_1$) be the vertex set of the component of 
$\Tt_{\Sigma}\backslash v_1$ 
(resp.\ $\Tt_{\Sigma}\backslash v_2$) that contains $v_2$ 
(resp.\ $v_1$). Then $\phi^{-1}(V_2)\in \pi(v_1)$ and 
$\phi^{-1}(V_1)\in \pi(v_2)$, and $\phi^{-1}(V_2)\cup \phi^{-1}(V_1)=X$. 
Thus $\Pi_s$ is strongly compatible. This completes the proof that
(i) implies (iii) and thus the proof 
of the 
equivalence of (i)--(iii).

To establish the uniqueness part of the theorem, let $\Pi$ be a strongly 
compatible partition system on $X$ such that $\Sigma_{\Pi}=\Sigma$. 
Let $l$ be a leaf of $\Tt_{\Sigma}$ and let $u$ be the unique vertex 
of $\Tt_{\Sigma}$ 
adjacent to $l$. Since $\Tt_{\Sigma}$ is even, it follows by 
Lemma~\ref{littleeven}(ii) that $u$ is odd
and so $\phi^{-1}(u)=\emptyset$.
We next show that $\pi(u)\in \Pi$.

Suppose that $\pi(u)\not\in \Pi$. Let $\pi(u)=\{A_1, A_2, \ldots, A_t\}$, 
where $t\ge 2$ and, for all $i\in \{1,\ldots,t\}$, 
denote the edge $e$ of $\Tt_{\Sigma}$ incident 
with $u$ such that $\sigma_e=A_i| X-A_i$ holds by $e_i$.
Without loss of generality, we may assume that 
$A_1=\phi^{-1}(l)$. By Lemma~\ref{labelvertex}, 
there is a partition $\pi_1\in \Pi$ such that $A_1\in \pi_1$. 
Consider $\pi_1$. Since $A_1\in \pi_1$, it follows by 
Lemma~\ref{displaypi}(ii) that each path joining $l$ to another 
leaf of $\Tt_{\Sigma}$ contains exactly two edges of any subset 
$E_{\pi_1}$ of edges of $\Tt_{\Sigma}$ displaying $\pi_1$. As no 
other part of $\pi_1$ contains $A_1$, it follows that $\pi_1$ is a 
refinement\footnote{A partition $\pi'$ of $X$ is called 
a {\em refinement} of a partition
$\pi$ on $X$ if every part of $\pi'$ is a subset of a part of $\pi$.} 
of $\pi(u)$. Since $\pi(u)\not\in \Pi$ and thus
$\pi_1\not=\pi(u)$, this implies that, 
for some $i\in \{2, 3, \ldots, t\}$, the part $A_i$
is the disjoint union 
of at least two parts in $\pi_1$. Without loss of generality,
we may assume that $i=t$. Since $A_t$ 
is the disjoint union of at
least two parts in $\pi_1$, there is a partition, $\pi_2$ say, in $\Pi$ 
with $\pi_2$ distinct 
from $\pi_1$ such that a subset $E_{\pi_2}$ of edges of $\Tt$ that displays 
$\pi_2$ contains $e_t$. In particular, either 
$A_t\in \pi_2$ or $(X-A_t)\in \pi_2$.

We first show that $A_t\not\in \pi_2$. Assume that 
$A_t\in \pi_2$ holds. Then independent of the size of $\pi_2$ we must have
that the degree of $u$ cannot be two as otherwise
$\pi_2=\pi(u)$ would follow; a contradiction. We next distinguish
between $|\pi_2|\geq 3$ and $|\pi_2|=2$. If $|\pi_2|\geq 3$ then there
exists some $B\in \pi_2$ distinct from $A_t$ such that 
$\phi^{-1}(l)\subseteq B$. Let $e_B$ denote an edge
of $\Tt$ that displays the split $B|(X-B)$ which must exist as 
$\pi_2\in \Pi$. Note that $B\not=A_1$ as otherwise, since  
$A_1\in\pi_1$, $B\in \pi_2$ and $\pi_1\not=\pi_2$,
the
multiplicity of the split $B|X-B$ in $\Sigma_{\Pi}$ is at least two. 
But then the degree of $u$ is two which is impossible. 
Consequently, $e_B\not=e_1$. Moreover since $|\pi_2|\geq 3$ it follows that
$B=A_t$ or $B=X-A_t$ cannot hold either and so
$e_B\not=e_t$.  Thus the path from
$l$ to any vertex $a$ of $\Tt_{\Sigma}$ with $\phi^{-1}(a)\subseteq A_t$ 
holding does not cross the edge $e_B$. Combined with the
fact that $A_1=\phi^{-1}(l)\subseteq B$ it follows that 
$A_1\cup A_t\subseteq B$  which is impossible as $A_t$ and $B$ are
distinct parts of $\pi_2$. Thus, $|\pi_2|\geq 3$ cannot hold.
If $|\pi_2|=2$ then $\pi_2=\{A_t,B\}$ and so $\pi_1$ is a refinement of
$\pi_2$. But then $\pi_1$ and $\pi_2$ cannot be strongly compatible;
a contradiction. Thus,  $|\pi_2|=2$ cannot hold either. 
Consequently, $A_t\not\in \pi_2$, as required.

Now assume that $(X-A_t)\in \pi_2$. Since, as seen above, 
$A_t\not\in \pi_2$ it follows that
$A_t$ is the disjoint union of at least two parts in $\pi_2$. By
the choice of $A_t$ as the union of at least two parts in $\pi_1$ 
it follows that $\pi_1$ and $\pi_2$ are not strongly compatible; a 
contradiction. Hence $\pi(u)\in \Pi$ as required.

We complete the uniqueness part of the proof using induction on 
$k=|\Sigma|$. If $k=2$, then there is exactly one 
partition system $\Pi$ such that $\Sigma_{\Pi}=\Sigma$, and the 
uniqueness result follows. Now suppose that $k\ge 3$ and the uniqueness 
result holds for all compatible split systems $\Sigma'$ on $X$ for which 
$\Tt_{\Sigma'}$ is an even $X$-tree and $|\Sigma'|\le k-1$.

Let $\Pi$ be a strongly compatible partition system on $X$ such that 
$\Sigma_{\Pi}=\Sigma$. Let $l$ be a leaf of $\Tt_{\Sigma}$ and let 
$u$ be the vertex of $\Tt_{\Sigma}$ adjacent to $l$. By above, 
$\pi(u)\in \Pi$. Let $\Sigma'=\Sigma-\Sigma_{\pi(u)}$. Then $\Sigma'$ 
is compatible, $|\Sigma'|\le k-1$, and $\Tt_{\Sigma'}$ is an even 
$X$-tree as it corresponds to the weak $X$-tree obtained
from $\Tt_{\Sigma}$ by contracting all edges incident with
$u$ and labelling the resulting vertex with the union of the label sets 
of the vertices previously adjacent to $u$. 
Therefore, by the induction assumption,
$$
\Pi'_s=\{\pi(v): v\in V_{\mbox{\scriptsize odd}}(\Tt_{\Sigma'})\},
$$
is the unique strongly compatible partition system 
on $X$ for which $\Sigma_{\Pi'_s}=\Sigma'$. Therefore, as
$$
V_{\mbox{\scriptsize \rm odd}}-V_{\mbox{\scriptsize odd}}(\Tt_{\Sigma'})
=\{u\},
$$
it follows that $\Pi=\Pi'_s\mcup \{\pi(u)\}=\Pi_s$. Thus the uniqueness 
property holds for $\Sigma$. This completes the proof of the theorem.
\end{proof}

Let $\Tt$ be a weak $X$-tree and let $e$ be an edge of $\Tt$. 
We denote by $\Tt/e$ the weak $X$-tree obtained from $\Tt$ by 
contracting $e$ and labelling the new identified vertex with the 
union of the labels of the end vertices of $e$. If $F$ is a subset 
of the edges of $\Tt$, then $\Tt/F$ denotes the weak $X$-tree obtained 
from $\Tt$ by contracting each of the edges in $F$ in this way where of
course the order of contraction is of no relevance. 

The next result sheds light into the structure of
weak $X$-trees obtained from even $X$-trees by contracting
edges. Its proof follows from 
Lemma~\ref{displaypi}(ii) and is omitted.

\begin{lemma}
Let $\Tt$ be an even $X$-tree and let $\pi$ be a partition of $X$ 
displayed by $\Tt$. Let $F$ be a subset of edges of $\Tt$ that 
displays $\pi$. Then $\Tt/F$ is an even $X$-tree. 
\label{tinyeven}
\end{lemma}

The following corollary may be viewed as the 
converse of Lemma~\ref{displaypi}(ii). 

\begin{corollary}
Let $\Tt$ be an even $X$-tree, and let $F$ be a non-empty
subset of edges of $\Tt$ 
with the property that, for every pair of labelled vertices $u$ and $v$, 
the path joining $u$ and $v$ contains exactly~$0$ or~$2$ edges of $F$. 
Then there is a partition system $\Pi\in{\mathbb P}(\Tt)$ and 
a partition $\pi\in \Pi$ such that $F$ displays $\pi$.
\label{two}
\end{corollary}

\begin{proof}
Let $F=\{f_1, f_2, \ldots, f_t\}$, $t\geq 2$, 
and let $i\in \{1,2, \ldots, t\}$. Let $\Tt=(T, \phi)$ and 
consider $\Tt\backslash f_i$. Since there are exactly two edges 
in $F$ on the path between a leaf in one component 
of $\Tt\backslash f_i$ and a leaf in the other component, 
one of the components contains no edges in $F$. For each $i$, 
let $V_i$ denote the vertex set of the component of $\Tt\backslash f_i$ 
containing no edges in $F$.

We now show that 
$$
\pi=\{\phi^{-1}(V_1), \phi^{-1}(V_2), \ldots, \phi^{-1}(V_t)\}
$$ 
is a partition of $X$. If not, then there is a labelled vertex, $w$ say, 
such that the component of $\Tt\backslash F$ that contains $w$ in its
vertex set is not contained in 
$\{V_1, V_2, \ldots, V_t\}$. But then, in the path from $w$ to a leaf 
in any one of the components $V_1, V_2, \ldots, V_t$, there is exactly 
one edge in $F$; a contradiction. Thus $\pi$ is a partition of $X$. 

To see that there is a partition system in ${\mathbb P}(\Tt)$ 
containing $\pi$, observe that, by Lemma~\ref{tinyeven}, $\Tt/F$ is an 
even $X$-tree and so, 
by Theorem~\ref{thm:tree:sc}, there is a partition system $\Pi'$ 
on $X$ such that $\Sigma_{\Pi'}=\Sigma(\Tt/F)$, that is,
$\Pi'\in{\mathbb P}(\Tt/F)$. By Lemma~\ref{lem:union}(i),
it now follows that $\Pi'\mcup \{\pi\}$ 
is a partition system in ${\mathbb P}(\Tt)$.
\end{proof}

\section{Maximum-Sized Partition Systems}
\label{maxsized}

In this section, for a compatible split system $\Sigma$ on $X$, we show 
that the unique strongly compatible partition system in 
$\setpa(\Sigma)$ 
is a partition system in $\setpa(\Sigma)$ of maximum size. 

We begin by proving a lemma
for which we require some additional notation.
Let $T$ be a tree with at least two leaves. We denote
by $V_{\mbox{\scriptsize \rm int}}(T)$ the set of interior 
vertices of $T$. Suppose ``odd'' and ``even'' are the colours
of a $2$-colouring of $T$.
Extending our notation for even $X$-trees, we denote the sets of
vertices of $T$ coloured ``odd'' and ``even'' by 
$V_{\mbox{\scriptsize \rm odd}}(T)$
and $V_{\mbox{\scriptsize \rm even}}(T)$, respectively. 
Furthermore, we denote the sets of interior vertices of $T$ coloured 
 ``odd" and  ``even" by 
$(V_{\mbox{\scriptsize \rm int}})_{\mbox{\scriptsize \rm odd}}(T)$
and $(V_{\mbox{\scriptsize \rm int}})_{\mbox{\scriptsize \rm even}}(T)$, 
respectively.

\begin{lemma}
Let $T$ be a tree with at least two leaves and 
suppose we have a $2$-colouring of the vertex 
set of $T$ using the set $\{\rm odd, \rm even\}$. Then
$$
|V_{\mbox{\scriptsize \rm odd}}(T)|
\ge |(V_{\mbox{\scriptsize \rm int}})_{\mbox{\scriptsize \rm even}}(T)|+1.
$$
\label{2colour}
\end{lemma}

\begin{proof}
The proof is by induction on the size $m$ of the vertex  set of $T$. 
If $m=2$, then a routine check shows that the lemma holds. Now suppose 
that $m\ge 3$ and the result holds for all trees with fewer than $m$ 
vertices. Let $v$ be a leaf of $T$ and let $T'$ be the tree obtained from $T$ 
by deleting $v$ and the edge incident with it. For ease of presentation, set 
$V_{\mbox{\scriptsize odd}}=V_{\mbox{\scriptsize odd}}(T)$, 
$V'_{\mbox{\scriptsize odd}}=V'_{\mbox{\scriptsize odd}}(T)$, 
$(V_{\mbox{\scriptsize int}})_{\mbox{\scriptsize even}}=
(V_{\mbox{\scriptsize int}})_{\mbox{\scriptsize \rm even}}(T)$, and 
$(V'_{\mbox{\scriptsize int}})_{\mbox{\scriptsize even}}=
(V_{\mbox{\scriptsize int}})_{\mbox{\scriptsize \rm even}}(T')$. Since 
$|V(T')|<m$ and the given 
$2$-colouring of $T$ induces a $2$-colouring of $T'$, it follows by the 
induction assumption that
\begin{align}
|V'_{\mbox{\scriptsize odd}}|\ge 
|(V'_{\mbox{\scriptsize int}})_{\mbox{\scriptsize even}}|+1.
\label{coloureqn}
\end{align}

Let $t$ and $t'$ denote the size of the leaf sets of $T$ 
and $T'$, respectively, and let $u$ 
denote the unique vertex adjacent to $v$ in 
$T$. We divide the 
rest of the proof into two cases depending upon whether $t'=t-1$, 
in which case the degree of $u$ in $T$ is at least three, or $t'=t$, 
in which case the degree of $u$ in $T$ is two. If $t'=t-1$. Then 
$V_{\mbox{\scriptsize int}}(T)=V_{\mbox{\scriptsize int}}(T')$. 
Therefore, 
by (\ref{coloureqn}),
\begin{align*}
|V_{\mbox{\scriptsize \rm odd}}|\ge |V'_{\mbox{\scriptsize odd}}|
\ge |(V'_{\mbox{\scriptsize int}})_{\mbox{\scriptsize even}}|+1
=|(V_{\mbox{\scriptsize int}})_{\mbox{\scriptsize even}}|+1,
\end{align*}
and the lemma holds.

Now suppose that $t'=t$. If $v$ is coloured even, 
then $(V_{\mbox{\scriptsize int}})_{\mbox{\scriptsize even}}
=(V'_{\mbox{\scriptsize int}})_{\mbox{\scriptsize even}}$ 
and so, by (\ref{coloureqn}),
\begin{align*}
|V_{\mbox{\scriptsize \rm odd}}|
\ge |V'_{\mbox{\scriptsize odd}}|
\ge |(V'_{\mbox{\scriptsize int}})_{\mbox{\scriptsize even}}|+1
=|(V_{\mbox{\tiny int}})_{\mbox{\scriptsize even}}|+1.
\end{align*}
So we may assume that $v$ is coloured odd. Then 
$|(V_{\mbox{\scriptsize int}})_{\mbox{\scriptsize even}}|=
|(V'_{\mbox{\scriptsize int}})_{\mbox{\scriptsize even}}|+1$ 
and $|V_{\mbox{\scriptsize \rm odd}}|=|V'_{\mbox{\scriptsize odd}}|+1$. 
Combining these two equations with (\ref{coloureqn}), it follows that
\begin{align*}
|V_{\mbox{\scriptsize \rm odd}}|=|V'_{\mbox{\scriptsize odd}}|+1
\ge |(V'_{\mbox{\scriptsize int}})_{\mbox{\scriptsize even}}|+2
=|(V_{\mbox{\scriptsize int}})_{\mbox{\scriptsize even}}|+1.
\end{align*}
This completes the proof of the lemma.
\end{proof}

Denoting for a vertex $v$ of a graph the degree of $v$ by $\deg(v)$
we are now ready to give the aforementioned characterization.

\begin{theorem}
\label{thm:maximum}
Let $\Pi$ be a compatible partition system on $X$ and 
let $\Pi_s$ be the unique strongly compatible partition system 
in $\setpa(\Sigma_{\Pi})$.
Then $|\Pi'| \le |\Pi_s|$, for all $\Pi'\in\setpa(\Sigma_{\Pi})$.
\label{thm:tc:sc:bound}
\end{theorem}

\begin{proof}
Let $\Pi'\in \mathbb P(\Sigma_{\Pi})$ and $ \pi_1\in\Pi'$.
Let $\Pi'_1$ denote $\Pi'-\{\pi_1\}$.
Since $\Tt_{\Pi}\cong \Tt_{\Pi'}$ and, by Lemma~\ref{lem:union}(ii),
 $\Sigma_{\Pi'}=\Sigma_{\Pi-\{\pi_1\}}=\Sigma_{\Pi}-\Sigma_{\pi_1}$
holds it follows that 
$\Tt_{\Pi'_1}\cong \Tt_{\Pi}/ E_{\pi_1}$, 
where $E_{\pi_1}$ is a subset of edges of $\Tt_{\Pi}$ that 
displays $\pi_1$. Since, by Theorem~\ref{thm:tree:sc},  
$\Tt_{\Pi}$ is even,  
Lemma~\ref{tinyeven} implies that $\Tt_{\Pi'_1}$ is even. 
Put $V_{\mbox{\scriptsize odd}}=V_{\mbox{\scriptsize odd}}(\Tt_{\Pi})$
and $(V'_1)_{\mbox{\scriptsize odd}}= 
V_{\mbox{\scriptsize odd}}(\Tt_{\Pi'_1})$. We next show that
\begin{align}
|V_{\mbox{\scriptsize \rm odd}}|\ge |(V'_1)_{\mbox{\scriptsize odd}}|+1
\label{oddeqn}
\end{align}
holds which will be crucial for an inductive argument on the edge set of
$\Tt_{\Pi}$ which will allow us to establish the theorem.

To observe~(\ref{oddeqn}), let $V_1, V_2, \ldots, V_t$ denote 
the vertex sets of the components of 
$\Tt_{\Pi}\backslash E_{\pi_1}$. By Lemma~\ref{displaypi}(i), 
precisely one of these vertex sets has the property that no vertex 
is labelled. Without loss of generality, we may assume that this vertex set 
is $V_t$. We consider two cases depending on the size of $V_t$. Suppose 
first that 
$|V_t|=1$, and let $V_t=\{u\}$. If $u$ is odd, then each of the $\deg(u)$ 
vertices adjacent to $u$ is even, and it follows that 
$(V'_1)_{\mbox{\scriptsize odd}}$ has exactly one less vertex than 
$V_{\mbox{\scriptsize \rm odd}}$. In particular, (\ref{oddeqn}) holds. 
If $u$ is even, then each of the $\deg(u)$ vertices adjacent to $u$ is
odd. Therefore
\begin{align*}
|V_{\mbox{\scriptsize \rm odd}}|=|(V'_1)_{\mbox{\scriptsize odd}}| 
+ \deg(u) -1.
\end{align*}
But $\deg(u)-1\ge 1$ as $\deg(u)\ge 2$, and so (\ref{oddeqn}) holds.

Now suppose that $|V_t|\ge 2$. Let $\Tt_t$ be the subtree of $\Tt_{\Pi}$ 
induced by $V_t$ and let 
$\Tt^+_t$ be the subtree of $\Tt_{\Pi}$ whose edge set is precisely 
$E(\Tt_t)\cup E_{\pi_1}$. Let 
$(V_t)_{\mbox{\scriptsize even}}=V_{\mbox{\scriptsize even}}(\Tt_t)$
 and 
$(V^+_t)_{\mbox{\scriptsize odd}}=V_{\mbox{\scriptsize odd}}(\Tt^+_t)$
 Then
\begin{align*}
|(V'_1)_{\mbox{\scriptsize odd}}| = |(V_t)_{\mbox{\scriptsize even}}| 
+ \big(|V_{\mbox{\scriptsize odd}}| - |(V^+_t)_{\mbox{\scriptsize odd}}|\big),
\end{align*}
and therefore
\begin{align*}
|V_{\mbox{\scriptsize \rm odd}}|-|(V'_1)_{\mbox{\scriptsize odd}}|=
|(V^+_t)_{\mbox{\scriptsize odd}}|-|(V_t)_{\mbox{\scriptsize even}}|.
\end{align*}
Since 
$(V_{\mbox{\scriptsize \rm int}})_{\mbox{\scriptsize \rm even}}(\Tt^+_t)=
(V_t)_{\mbox{\scriptsize \rm even}}$, we have
$|(V^+_t)_{\mbox{\scriptsize odd}}|-|(V_t)_{\mbox{\scriptsize even}}|\ge 1$ 
by Lemma~\ref{2colour}
and so (\ref{oddeqn}) follows.

Having established (\ref{oddeqn}), we complete the proof of the theorem 
by induction on the size of the edge set $E_{\Pi}$ of ${\mathcal T}_{\Pi}$. 
If $|E_{\Pi}|=2$, then $\Pi$ is the unique partition system in 
${\mathbb P}(\Sigma_{\Pi})$. In particular, $\Pi$ is the unique 
strongly compatible partition system in ${\mathbb P}(\Sigma_{\Pi})$ 
and so the theorem holds. Now assume that the theorem holds for all 
compatible partition systems whose corresponding even $X$-tree has fewer 
edges than ${\mathcal T}_{\Pi}$. Let $\Pi'$, $\pi_1$, and $\Pi'_1$ be as 
defined at the beginning of the proof. Then, as observed there, 
$\Tt_{\Pi_1'}$ must be even. By 
Theorem~\ref{thm:tree:sc}, ${\mathbb P}(\Sigma_{\Pi_1'})$
must contain a unique strongly compatible partition system 
$(\Pi'_1)_s$ on $X$. But then
$|\Pi'_1| \le |(\Pi'_1)_s|$, by induction assumption. Combined with
(\ref{oddeqn}) and Theorem~\ref{thm:tree:sc} which
implies that $|V_{\mbox{\scriptsize odd}}|=|\Pi_s|$
and $|(V'_1)_{\mbox{\scriptsize odd}}|= |(\Pi_1')_s|$ hold,
we obtain
\begin{align*}
|\Pi'| = |\Pi'_1|+1 \le |(\Pi'_1)_s| + 1 
=|(V'_1)_{\mbox{\scriptsize odd}}| + 1 \le |V_{\mbox{\scriptsize odd}}| 
= |\Pi_s|.
\end{align*}
This completes the proof of the theorem.
\end{proof}

As we have seen in the example presented in the
introduction, for a compatible split system $\Sigma$
with ${\mathbb P}(\Sigma) \neq \emptyset$, the strongly 
compatible partition system in ${\mathbb P}(\Sigma)$ is not necessarily 
the only partition system in ${\mathbb P}(\Sigma)$ of maximum size.
For such $\Sigma$, it could therefore be of interest to try to characterize the
set of partition systems in ${\mathbb P}(\Sigma)$ of maximum size.
In regards to this, it is worth noting that in case
$\Sigma$ is a compatible set of splits
corresponding to a phylogenetic $X$-tree with all interior 
vertices of degree three, then 
it is not difficult to show that there
is a unique partition system in ${\mathbb P}(\Sigma)$ of 
maximum size, namely the strongly compatible partition system.


\section{Constructing Minimum-Sized Partition Systems}
\label{minsized}

We now turn our attention to the problem
of understanding minimum elements in the set 
$\setpa(\Sigma)$ 
for a compatible split system $\Sigma$.
More specifically, we construct, for an even $X$-tree $\Tt$,  
a {\em ${\mathbb P}(\Tt)$-minimum} partition system on $X$,
that is, a partition system  $ \Pi$ on $X$
such that $\Sigma(\mathcal T)=\Sigma_{\Pi}$
and $\Pi$ is of minimum size with respect to this property. 
The construction is 
presented in the form of the \textsc{MinSizePartition} algorithm
in Fig.~\ref{algorithm:smallest-partition-system}. 
It will make use of the following decomposition of a weak $X$-tree.

Let ${\mathcal T}=(T;\phi)$ be a weak $X$-tree with edge set $E$ and let 
$v$ be a labelled interior vertex of $\mathcal T$. Suppose that $v$ has 
degree~$k\geq 2$. Now partition $E$ so that, for all edges $e$ and $f$, 
we have $e$ and $f$ in the same part if and only if the path from $e$ to 
$f$ in $\mathcal T$ avoids $v$. Let $\{E_1, E_2, \ldots, E_k\}$ denote the 
resulting partition on $E$. For each $i\in[k]=\{1,\dots,k\}$, $k \ge 2$, 
let $e_i$ denote the unique edge in $E_i$ 
incident with $v$ in $\mathcal T$ and let $A_i|B_i$ denote the $X$-split 
corresponding to $e_i$, where $\phi^{-1}(v)\subseteq B_i$. For each 
$i\in [k]$, 
let ${\mathcal T}_i$ denote the weak $X$-tree induced by $E_i$, where 
the label of every vertex of $\mathcal T$ is retained except for 
$v$ whose label changes to $B_i$. The collection 
$\{{\mathcal T}_1, {\mathcal T}_2, \ldots, {\mathcal T}_k\}$ is called 
the {\em decomposition of $\mathcal T$ with respect to $v$} and is denoted 
by ${\mathcal D}({\mathcal T}, v)$. To illustrate the decomposition, consider 
the even $X$-tree $\mathcal T$ shown in 
Fig.~\ref{fig:decom}, where $X=\{1, 2, 3, 4, 5, 6, 7\}$. The decomposition 
of $\mathcal T$ with respect to the vertex labelled $3$ is shown in the 
right of that figure.

\begin{figure}[h]
\center
\input{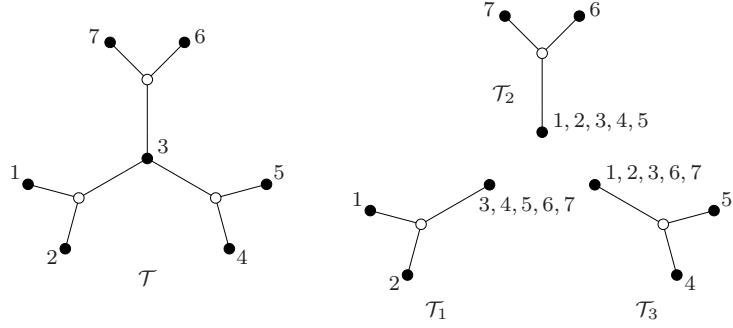}
\caption{A weak $X$-tree $\mathcal T$ with $X=\{1, 2, \ldots, 7\}$ (left) and
the decomposition $\{\mathcal T_1, \mathcal T_2, \mathcal T_3\}$ of $\mathcal
T$ with respect to the vertex labelled $3$ (right). Filled vertices denote
labelled vertices.}
\label{fig:decom}
\end{figure}

Two observations that we freely use in the rest of this section
 are the following. First,
\begin{align*}
\Sigma({\mathcal T}) = \Sigma({\mathcal T}_1)\uplus 
\Sigma({\mathcal T}_2)\uplus \cdots\uplus \Sigma({\mathcal T}_k)
\end{align*}
and, for all distinct $i,j\in [k]$, we have 
$\Sigma(\mathcal T_i)\cap \Sigma(\mathcal T_j)=\emptyset$. Second, 
if $\mathcal T$ is an even $X$-tree, then each of the weak $X$-trees
${\mathcal T}_1, {\mathcal T}_2, \ldots, {\mathcal T}_k$ is even.

The next lemma will be used later in this section.

\begin{lemma}
Let $\Tt$ be a weak $X$-tree, and let $v$ be a 
labelled interior vertex 
of ${\mathcal T}$. Let ${\mathcal D}({\mathcal T}, v)=
\{{\mathcal T}_1, {\mathcal T}_2, \ldots, {\mathcal T}_k\}$ and 
let $\Pi$ be a partition system on $X$. Then
$\Pi\in {\mathbb P}(\mathcal T)$ if and only 
if there is a partition $\{\Pi_1, \Pi_2, \ldots, \Pi_k\}$ of 
$\Pi$ such that, for all $i\in[k]$, we 
have $\Pi_i\in {\mathbb P}(\mathcal T_i)$. Moreover, 
if $\Pi\in \setpa(\mathcal T)$, then such a partition of $\Pi$ 
is unique.
\label{lem:decomp}
\end{lemma}

\begin{proof}
Suppose first that there is a partition $\{\Pi_1, \Pi_2, \ldots, \Pi_k\}$ 
of $\Pi$ such that, for all $i\in[k]$, we have 
$\Pi_i\in {\mathbb P}(\mathcal T_i)$. Then, as
$$
\Sigma({\mathcal T}) = \Sigma({\mathcal T}_1) 
\uplus \Sigma({\mathcal T}_2) \uplus \cdots 
\uplus \Sigma({\mathcal T}_k)
$$
and $\Sigma_{\Pi_i}= \Sigma(\Tt_i)$ holds for all $i\in[k]$, 
Lemma~\ref{lem:union} implies
$$
\Pi = \Pi_1 \uplus \Pi_2 \uplus \cdots \uplus \Pi_k 
\in {\mathbb P}(\mathcal T).
$$

Conversely,  suppose that $\Pi\in {\mathbb P}(\mathcal T)$. 
For each $i\in[k]$, let $E_i$ denote the edge set of ${\mathcal T}_i$. 
Let $\pi\in \Pi$ and let $E_{\pi}$ be a subset of edges of 
$\mathcal T$ that displays $\pi$. If $E_{\pi}$ contains
distinct edges $e$ and $f$, then with $x\in e$ and $y\in f$ such that
$x$ and $y$ lie on the path from $a\in e-\{x\}$ to $b\in f-\{y\}$,
it is easily seen that the path 
from $x$ to $y$ avoids 
the labelled vertex $v$. In particular, $E_{\pi}\subseteq E_i$ for some 
$i\in[k]$. Furthermore, as 
$\Sigma({\mathcal T}_i)\cap \Sigma({\mathcal T}_j)=\emptyset$ for all 
distinct $i,j\in [k]$, there is a unique $i^*\in[k]$ for which 
$E_{\pi}\subseteq E_{i^*}$. Now let $\{\Pi_1, \Pi_2, \ldots, \Pi_k\}$ denote 
the unique partition of $\Pi$ such that, for all $i\in[k]$, we have 
$\Sigma_{{\Pi}_i}\subseteq \Sigma({\mathcal T}_i)$. But 
$\Sigma_{\Pi}=\Sigma({\mathcal T})$ and so, for all $i\in[k]$, we have 
$\Sigma_{{\Pi}_i}=\Sigma({\mathcal T}_i)$, that is,
$\Pi_i\in {\mathbb P}(\mathcal T_i)$. This completes the proof of the lemma.
\end{proof}

For an even $X$-tree $\Tt$, we next present
our construction \textsc{MinSizePartition} 
in the form of pseudo-code and establish its correctness 
in Theorem~\ref{theo:verify}. 
For example, for the even $X$-tree $\Tt$ depicted in Fig.~\ref{trees}
the ${\mathbb P}(\Tt)$-minimum partition system
that we construct is the partition system $\Pi_6$
given in the introduction. 

For a weak $X$-tree $\mathcal T=(T; \phi)$ in which all 
interior vertices are unlabelled, set $\pi_{\min}(\mathcal T)$ 
to be the partition
$$
\pi_{\min}({\mathcal T})=\{\phi^{-1}(v): \mbox{$v$ is a leaf 
of $\mathcal T$}\}
$$
of $X$. Note that 
$\Sigma_{\pi_{\min}({\mathcal T})}\subseteq \Sigma({\mathcal T})$
and that for the even $X$-tree $\Tt$ depicted in Fig.~\ref{trees}
we have $\pi_{\min}({\mathcal T})=\{1|2|3|4|5|6\}$.

\begin{figure}[h]
\centering
\parbox{0cm}{\begin{tabbing}
XX\= XX\= XX\=  XX\= XX\= XX\= XX\=  XXXXX\=  XXXXXXXXXXX\= \kill \\
{\textsc{MinSizePartition}($\mathcal T$)} \\
\rule{\columnwidth}{0.5pt}\\
Input: \> \> \> An even $X$-tree $\mathcal T$. \\
Output: \> \> \> A  partition system $\Pi_{\min}({\mathcal T})$ on $X$ that is 
${\mathbb P}({\mathcal T})$-minimum. \\
\rule{\columnwidth}{0.5pt} \\
If there exists an interior vertex $v$ in $\mathcal T$ that is labelled \\
\> Construct the decomposition ${\mathcal D}({\mathcal T},v)$, 
say $\{{\mathcal T}_1, {\mathcal T}_2, \dots, {\mathcal T}_k\}$, of 
$\mathcal T$ \\
\> For each $i\in [k]$, call 
{\sc MinSizePartition}($\mathcal T_i$) \\
\> Return $\Pi_{\min}({\mathcal T})\leftarrow \Pi_{\min}({\mathcal T}_1)\mcup 
\Pi_{\min}({\mathcal T}_2)\mcup \cdots\mcup \Pi_{\min}({\mathcal T}_k)$ \\
Else, set $\pi_{\min}=\pi_{\min}({\mathcal T})$ and 
set $\Sigma'=\Sigma({\mathcal T})-\Sigma_{\pi_{\min}}$ \\
\> If $\Sigma'$ is non-empty \\
\> \> Construct the even $X$-tree $\mathcal T'$ for which 
$\Sigma({\mathcal T'})=\Sigma'$ \\
\> \> Call {\sc MinSizePartition}($\mathcal T'$) \\
\> \> Return $\Pi_{\min}({\mathcal T})\leftarrow \{\pi_{\min}\}\mcup 
\Pi_{\min}({\mathcal T}')$ \\
\> Else \\
\> \> Return $\Pi_{\min}({\mathcal T})\leftarrow \{\pi_{\min}\}$ \\
\> Endif \\
Endif \\
\end{tabbing}}
\caption{Pseudo-code for \textsc{MinSizePartition}.}
\label{algorithm:smallest-partition-system}
\end{figure}

To establish the correctness of \textsc{MinSizePartition}, we make use of 
the next lemma.

\begin{lemma}
Let $\mathcal T$ be an even $X$-tree with no labelled interior 
vertices. Then there exists a ${\mathbb P}({\mathcal T})$-minimum 
partition system that contains $\pi_{\min}({\mathcal T})$.
\label{lem:min}
\end{lemma}

\begin{proof}
For convenience, set $\pi_{\min}=\pi_{\min}({\mathcal T})$. If 
$A\in \pi_{\min}$, then, by Lemma~\ref{labelvertex}, each partition 
system in $\setpa(\mathcal T)$ contains a partition $\pi$ with $A\in \pi$. 
Suppose that $\Pi$ is a ${\mathbb P}({\mathcal T})$-minimum partition system. 
We may assume that $\pi_{\min}\not\in \Pi$. Let $\Pi'$ be a minimum-sized 
subset of $\Pi$ such that, for each $A\in \pi_{\min}$, there is a partition 
$\pi'$ in $\Pi'$ with $A\in \pi'$. Note that $\Sigma_{\Pi'}=\Sigma(\Tt)$
need not hold. Without loss of generality, we may assume 
that, $\Pi'$ is a minimum-sized partition system contained in a 
${\mathbb P}({\mathcal T})$-minimum partition system with this property. 
We break the proof into two cases depending upon whether or not $\Pi'$ is 
a strongly compatible partition system on $X$.

First suppose that $\Pi'$ is strongly compatible. Then,
$\Sigma_{\Pi'}$ is compatible and so Theorem~\ref{thm:tree:sc}
implies that $\Tt_{\Sigma_{\Pi'}}$ is even.  Let $F$ denote the subset of
edges of $\Tt_{\Sigma_{\Pi'}}$ that are incident with some leaf of 
$\Tt_{\Sigma_{\Pi'}}$. Then, for any two leaves $u$ and $v$ of 
$\Tt_{\Sigma_{\Pi'}}$, the path between $u$ and $v$ contains
either $0$ or precisely two edges in $F$. Hence, by Corollary~\ref{two},
 there exists a 
partition system $\Pi''$ in ${\mathbb P}(\Sigma_{\Pi'})$ and a partition
$\pi\in \Pi''$ that displays $F$. But now the definition of $\pi_{\min}$
implies that $\pi=\pi_{\min}$ and so $\pi_{\min}\in  \Pi''$.
Consider the partition system
\begin{align*}
\hat{\Pi}=(\Pi-\Pi')\uplus \Pi''.
\end{align*}
Since $\Sigma_{\Pi'}=\Sigma_{\Pi''}$, it follows that $\hat{\Pi}$ is 
in ${\mathbb P}({\mathcal T})$. As $\Pi'$ is strongly compatible, it 
follows by Theorem~\ref{thm:tc:sc:bound} that $|\Pi''|\le |\Pi'|$.
Since $\Pi'\subseteq \Pi$ and $\Pi$ is 
a ${\mathbb P}({\mathcal T})$-minimum partition 
system, it follows that $|\Pi|=|\hat{\Pi}|$ and so
$\hat{\Pi}$ is also a ${\mathbb P}({\mathcal T})$-minimum partition 
system. Observing that $\pi_{\min}\in \hat{\Pi}$ 
completes the proof of the case when 
$\Pi'$ is strongly compatible.

Now suppose that $\Pi'$ is not strongly compatible. Then 
there exist distinct partitions $\pi$ and $\pi'$ in $\Pi'$ that 
are not strongly compatible. This implies that $\pi\cup \pi'$ is a 
hierarchy. To see this, suppose that $\pi\cup \pi'$ is not a hierarchy. 
Then there exists $A\in \pi$ and $A'\in \pi'$ such that each of the sets 
$A\cap A'$, $A\cap (X-A')$, $(X-A)\cap A'$ is non-empty. Furthermore, 
$(X-A)\cap (X-A')$ is also non-empty as $A\cup A'\neq X$. But 
$\pi, \pi'\in \Pi$ and $\Pi$ is compatible, so at least one of these 
intersections is empty; a contradiction.

Since $\pi\cup \pi'$ is a hierarchy, it follows that, for each $A\in \pi$, 
either $A$ is a subset of a part in $\pi'$ or $A$ is the disjoint union of 
parts in $\pi'$. Similarly, for each $A'\in \pi'$, either $A'$ is a subset 
of a part in $\pi$ or $A'$ is the disjoint union of parts in $\pi$. It now 
follows that there is a partition system $\{\pi_1, \pi_2\}$ on $X$ such 
that $\Sigma_{\{\pi_1, \pi_2\}}=\Sigma_{\{\pi, \pi'\}}$ and, for all 
$B\in \pi_1$, we have that $B$ is a subset of a part in $\pi_2$. Let
\begin{align*}
\Pi''=(\Pi'-\{\pi, \pi'\})\uplus \{\pi_1\}.
\end{align*}

Clearly, $|\Pi''|=|\Pi'|-1$. Furthermore, for each $A\in \pi_{\min}$, 
there exists, by assumption, a partition in $\Pi'$ containing $A$, 
and so, for each 
$A\in \pi_{\min}$, there is a partition in $\Pi''$ containing $A$. 
Now consider the partition system
\begin{align*}
\hat{\Pi}=(\Pi-\Pi')\uplus \Pi''\uplus \{\pi_2\}=(\Pi-\{\pi, \pi'\})
\uplus \{\pi_1, \pi_2\}.
\end{align*}
Since $\Sigma_{\Pi}=\Sigma_{\hat{\Pi}}$, it follows that $\hat{\Pi}$ 
is in ${\mathbb P}({\mathcal T})$. Therefore, as $\Pi$ is 
${\mathbb P}({\mathcal T})$-minimum, $\hat{\Pi}$ is 
${\mathbb P}({\mathcal T})$-minimum. But $|\Pi''| < |\Pi'|$ and 
$\Pi''$ is a subset of $\hat{\Pi}$ with the property that, for each 
$A\in \pi_{\min}$, there is a partition in $\Pi''$ containing $A$; 
a contradiction. This completes the proof of the case that $\Pi'$ is not 
strongly compatible.
\end{proof}

\begin{theorem}\label{theo:verify}
Let $\mathcal T$ be an even $X$-tree. Then the partition system 
$\Pi_{\min}(\mathcal T)$ returned
by \textsc{MinSizePartition} applied to $\mathcal T$ is a 
${\mathbb P}({\mathcal T})$-minimum
partition system.
\end{theorem}

\begin{proof}
We prove the theorem by induction on the number $m$ of interior 
vertices of $\mathcal T$.
Since $\mathcal T$ is even, $m\ge 1$. If $m=1$, then the unique 
interior vertex is adjacent
to each leaf of $\mathcal T$. It now follows by Lemma~\ref{labelvertex} 
combined with the definition of $\pi_{\min}({\mathcal T})$ that
$\{\pi_{\min}({\mathcal T})\}$ is the unique partition system in 
${\mathbb P}({\mathcal T})$,
and so {\sc MinSizePartition} correctly returns 
$\{\pi_{\min}({\mathcal T})\}$.

Let $m\ge 2$ and assume that {\sc MinSizePartition} 
correctly returns a ${\mathbb
P}({\mathcal T'})$-minimum partition system whenever it is 
applied to an even $X$-tree
${\mathcal T}'$ with fewer than $m$ interior vertices. We 
distinguish two cases depending
upon whether or not $\mathcal T$ has a labelled interior vertex.

First suppose that $\mathcal T$ has a labelled interior vertex $v$. 
Without loss of
generality, we may assume that at the first 
iteration of {\sc MinSizePartition} applied to
$\mathcal T$, the algorithm constructs the decomposition
$$
{\mathcal D}({\mathcal T}, v)=
\{{\mathcal T}_1, {\mathcal T}_2, \ldots, {\mathcal T}_k\}
$$
of $\Tt$ with respect to $v$ where $k$ is the degree of $v$.
Thus, to complete the proof of this case, it suffices to show that
\begin{align*}
\hat{\Pi} = \Pi_{\min}({\mathcal T}_1)\uplus \Pi_{\min}({\mathcal T}_2)
\uplus \cdots \uplus
\Pi_{\min}({\mathcal T}_k)
\end{align*}
is ${\mathbb P}({\mathcal T})$-minimum, where, for all 
$i\in [k]$,
$\Pi_{\min}({\mathcal T}_i)$ is the partition system on $X$ returned 
by {\sc MinSizePartition} applied to the even $X$-tree ${\mathcal T}_i$.

Let $i\in [k]$. Then, as ${\mathcal T}_i$ has fewer interior vertices than
$\mathcal T$, it follows by the induction assumption that 
$\Pi_{\min}({\mathcal T}_i)$ is
${\mathbb P}({\mathcal T}_i)$-minimum. One consequence of this fact is that
$\Pi_{\min}({\mathcal T}_i)$ is a partition system in 
${\mathbb P}({\mathcal T}_i)$. Combined with the definition of 
$\Pi$, Lemma~\ref{lem:decomp} implies that 
$\hat{\Pi}$ is a partition system in ${\mathbb P}({\mathcal T})$.
Now if $\hat{\Pi}$ is not ${\mathbb P}({\mathcal T})$-minimum, then 
there exists a partition
system $\Pi\in {\mathbb P}({\mathcal T})$ such that 
$|\Pi| < |\hat{\Pi}|$. By Lemma~\ref{lem:decomp}, there is a 
partition $\{\Pi_1, \Pi_2, \ldots, \Pi_k\}$ of $\Pi$ such
that, for all $i\in [k]$, we have $\Pi_i\in {\mathbb P}({\mathcal T}_i)$. 
But $|\Pi| < |\hat{\Pi}|$, and so there exists some $j\in [k]$ 
such that $|\Pi_j| < |\Pi_{\min}({\mathcal T}_i)|$ for some $i\in [k]$; 
a contradiction. Thus
$\hat{\Pi}$ is ${\mathbb P}({\mathcal T})$-minimum, as required.

Now suppose that $\mathcal T$ has no labelled interior vertex, and set
$\pi_{\min}=\pi_{\min}({\mathcal T})$ and 
$\Sigma'=\Sigma({\mathcal T})-\Sigma_{\pi_{\min}}$.
Then if $\Sigma'\not=\emptyset$ the algorithm
constructs the weak $X$-tree ${\mathcal T}'$ for which 
$\Sigma({\mathcal T}')=\Sigma'$. Note that $\Tt'\cong \Tt/E$ where
$E$ is a set of edges of $\Tt$ that displays $\pi_{\min}$
and so, since $\mathcal T$ is an even $X$-tree, it follows 
by Lemma~\ref{tinyeven} that $\mathcal T'$ is in fact an even $X$-tree.
For this case, it now  suffices to show that
\begin{align*}
\hat{\Pi} = \{\pi_{\min}\}\uplus \Pi_{\min}({\mathcal T}')
\end{align*}
is ${\mathbb P}({\mathcal T})$-minimum, where $\Pi_{\min}({\mathcal T}')$ 
is the partition
system on $X$ returned by {\sc MinSizePartition} applied
 to ${\mathcal T}'$. 

Since ${\mathcal T}'$ has fewer interior vertices than $\mathcal T$, 
it follows by the
induction assumption that $\Pi_{\min}({\mathcal T}')$ is 
${\mathbb P}({\mathcal
T}')$-minimum. This immediately implies that 
$\Pi_{\min}({\mathcal T}')$ is a partition
system in ${\mathbb P}({\mathcal T}')$, and so 
$\hat{\Pi}\in {\mathbb P}({\mathcal T})$. Now,
by Lemma~\ref{lem:min}, there is a ${\mathbb P}({\mathcal T})$-minimum 
partition system $\Pi$
containing $\pi_{\min}$. Let $\Pi'=\Pi-\{\pi_{\min}\}$. By 
Lemma~\ref{lem:union}, $\Pi'\in{\mathbb P}({\mathcal T}')$, and 
so $|\Pi_{\min}({\mathcal T}')|\le |\Pi'|$. Hence
\begin{align*}
|\hat{\Pi}| = |\Pi_{\min}({\mathcal T}')| + 1 \le |\Pi'| + 1 = |\Pi|.
\end{align*}
Thus, as $\Pi$ is ${\mathbb P}({\mathcal T})$-minimum, we deduce that 
$|\Pi|=|\hat{\Pi}|$ and so $\hat{\Pi}$ must also be
${\mathbb P}({\mathcal T})$-minimum, as required. 
This completes the proof of the second case and the
theorem.
\end{proof}

\section{Hierarchical Partition Systems}
\label{hsection}

In the previous section we showed how to construct, 
for an even $X$-tree $\Tt$, a ${\mathbb P}(\Tt)$-minimum 
partition system $\Pi$ on $X$. It appears to 
be a difficult problem to characterize the set 
of ${\mathbb P}(\Tt)$-minimum partition systems 
for arbitrary $\Tt$. However, in this section
we shall show that in case $\Tt$ contains
a vertex $\rho$ that has the same distance in  
$\Tt$ to every leaf, then we can characterize 
the $\setpa(\Tt)$-minimum partition systems 
(Theorem~\ref{hmin}). 
Note that such trees are sometimes called {\em equidistant} trees 
\cite[p.150]{SS03}.

The first result in this section shows that hierarchical partition 
systems are compatible. 

\begin{proposition}
Let $\Pi$ be a hierarchical partition system on $X$. Then 
$\Pi$ is compatible. Moreover, $\mathcal T_{\Pi}$ is isomorphic to 
$\mathcal T^-_{\rho}$, where $\mathcal T_{\rho}$ is the rooted weak 
$X$-tree with $\mathcal H(\mathcal T_{\rho})=\mcup_{\pi\in \Pi} \pi$.
\label{hcompatible}
\end{proposition}

\begin{proof}
By Theorem~\ref{hbuneman}, there is a unique rooted weak $X$-tree, 
${\mathcal T}_{\rho}$ say, with 
$\mathcal H(\mathcal T_{\rho})=\mcup_{\pi\in \Pi} \pi$. This 
implies that $\Sigma(\mathcal T^-_{\rho})=\Sigma_{\Pi}$. In particular, 
$\Pi$ is compatible and $\mathcal T_{\Pi}$ is isomorphic to 
$\mathcal T^-_{\rho}$.
\end{proof}

The next result gives some properties of 
$\mathcal T_{\Pi}$ in case $\Pi$ is a 
hierarchical partition system.

\begin{corollary}
Let $\Pi$ be a hierarchical partition system on $X$.
\begin{itemize}
\item[{\rm (i)}] If $u$ is an interior vertex of ${\mathcal T}_{\Pi}$, 
then $u$ is unlabelled.

\item[{\rm (ii)}] There is a vertex $\rho$ of ${\mathcal T}_{\Pi}$ such 
that, for all leaves $u$ and $v$,
$$d_{{\mathcal T}_{\Pi}}(\rho, u)=d_{{\mathcal T}_{\Pi}}(\rho, v)=|\Pi|.$$
\end{itemize}
\label{equal}
\end{corollary}

\begin{proof}
To prove (i), let $\mathcal T_{\Pi}=(T_{\Pi};\phi)$
and suppose that there is a labelled, interior vertex $u$ of 
${\mathcal T}_{\Pi}$. Let $A=\phi^{-1}(u)$. By Lemma~\ref{labelvertex}, 
for each edge incident with $u$, there is a distinct partition in $\Pi$ 
with a part that properly contains $A$. Since $u$ is an interior vertex, 
it has degree at least two, so there are at least two such partitions, 
$\pi_1$ and $\pi_2$ say. Let $A_1$ and $A_2$ be the parts of $\pi_1$ and 
$\pi_2$, respectively, that properly contain $A$. It is easily seen that 
neither $A_1\subseteq A_2$ nor $A_2\subseteq A_1$. But then, as 
$A\subseteq A_1\cap A_2$ and $A$ is non-empty, it follows that $\Pi$ 
is not hierarchical; a contradiction. This completes the proof of (i).

For the proof of (ii), let ${\mathcal T}_{\rho}=(T_{\rho}; \phi)$ 
be the rooted weak $X$-tree with root $\rho$ for which 
$\mathcal H(\mathcal T_{\rho})=\mcup_{\pi\in \Pi} \pi$. Let $u$ be a leaf of 
${\mathcal T}_{\rho}$. Now, the clusters displayed by the edges on the 
path from $\rho$ to $u$ are precisely the sets in $\mcup_{\pi\in \Pi} \pi$ 
containing $\phi^{-1}(u)$. Since each partition in $\Pi$ contains 
exactly one such set as a part, it follows that 
$d_{{\mathcal T}_{\rho}}(\rho, u)=|\Pi|$. By Proposition~\ref{hcompatible}, 
this in turn implies that
$$d_{{\mathcal T}_{\Pi}}(\rho, u)=d_{{\mathcal T}_{\Pi}}(\rho, v)$$
for all leaves $u$ and $v$ of ${\mathcal T}_{\Pi}$, thereby completing 
the proof of (ii).
\end{proof}

We now characterize the compatible split systems $\Sigma$
for which there exists some hierarchical partition system $\Pi$ 
with $\Sigma_{\Pi}=\Sigma$.

\begin{theorem}\label{rooted-equivalence}
Let $\Sigma$ be a compatible split system on $X$. Then there exists a 
hierarchical partition system $\Pi\in \setpa(\Sigma)$ 
if and only if ${\mathcal T}_{\Sigma}$ has a vertex $\rho$ such that, 
for all labelled vertices $u$ and $v$ of $\Tt_{\Sigma}$,
$$d_{{\mathcal T}_{\Sigma}}(\rho, u)=d_{{\mathcal T}_{\Sigma}}(\rho, v).$$
\label{hequal}
\end{theorem}

\begin{proof}
If there exists a hierarchical partition system $\Pi\in \setpa(\Sigma)$
then it follows by Corollary~\ref{equal} 
that ${\mathcal T}_{\Sigma}$ has a vertex $\rho$ such that, for all 
labelled vertices $u$ and $v$ of ${\mathcal T}_{\Sigma}$, we have 
$$
d_{{\mathcal T}_{\Sigma}}(\rho, u)=d_{{\mathcal T}_{\Sigma}}(\rho, v).
$$

To prove the converse, suppose that ${\mathcal T}_{\Sigma}$ has such a 
vertex $\rho$. Then no interior vertex of ${\mathcal T}_{\Sigma}$ is 
labelled. Let $d$ denote the distance from $\rho$ to a leaf of 
${\mathcal T}_{\Sigma}$. For each $i\in \{1,\dots,d\}$, let $E_i$ 
denote the subset of edges whose end vertex furthest from $\rho$ is 
distance $i$. Note that $\{E_1, E_2, \ldots, E_d\}$ is a partition of 
$E(\mathcal T_{\Sigma})$. Viewing ${\mathcal T}_{\Sigma}$ as a rooted 
weak $X$-tree with root $\rho$, let
$$\pi_i=\{C_e: e\in E_i\}$$
for each $i$. Since the leaves of ${\mathcal T}_{\Sigma}$ all have the
same distance to $\rho$, it follows that $\pi_i$ 
is a partition of $X$ for all $i$. In particular,
$$\Pi_h=\mcup_{i\in \{1,\dots,d\}} \pi_i$$
is a partition system on $X$ with $\Sigma_{\Pi_h}=\Sigma$. To see that 
$\Pi_h$ is hierarchical, let $A_i\in \pi_i$ and $A_j\in \pi_j$, where 
$\pi_i, \pi_j\in \Pi_h$. If $i=j$, then either $A_i\cap A_j=\emptyset$ or 
$A_i=A_j$. Thus we may assume that $i\neq j$. Without loss of generality, 
we may further assume that $i < j$. But then, again viewing 
${\mathcal T}_{\Sigma}$ as a rooted weak $X$-tree with root $\rho$, it 
is easily seen that either $A_i\cap A_j=\emptyset$ or $A_i\cap A_j=A_j$
as $A_i=C_{e_i}$ and $A_j=C_{e_j}$ for some $e_i\in E_i$ and some
$e_j\in E_j$ and $\mathcal H(\mathcal T_{\Sigma})=\mcup_{\pi\in \Pi_h} \pi$.
Consequently, $\Pi_h$ is hierarchical. This completes the proof of the 
converse, and thereby the proof of the theorem.
\end{proof}


We conclude this section by characterizing, for a 
compatible split system $\Sigma$  for which 
$\setpa(\Sigma)$ contains a hierarchical partition system,
the $\setpa(\Tc_{\Sigma})$-minimum partition systems.

\begin{theorem}
Let $\Sigma$ be a compatible split system on $X$ such that $\setpa(\Sigma)$ 
contains a hierarchical partition system, and let 
$\Pi\in \setpa(\Sigma)$. 
Then $\Pi$ is hierarchical if and only if $\Pi$ is 
$\setpa({\mathcal T}_{\Sigma})$-minimum.
\label{hmin}
\end{theorem}

\begin{proof}
Note that since $\setpa(\Sigma)$ contains a hierarchical partition
 system, it follows 
by Theorem~\ref{hequal} that $\mathcal T_{\Sigma}$ has a vertex $\rho$ such 
that, for all leaves $u$ and $v$ in $\mathcal T_{\Sigma}$,
$$
d_{{\mathcal T}_{\Sigma}}(\rho, u)=d_{{\mathcal T}_{\Sigma}}(\rho, v).
$$

First suppose that $\Pi$ is hierarchical. Then, since 
$\Pi\in \mathbb P(\Sigma)$ and so $\Tt_{\Sigma}\cong \Tt_{\Pi}$, 
Corollary~\ref{equal}(ii) implies
$$
\Delta({\mathcal T}_{\Sigma})=2d_{{\mathcal T}_{\Sigma}}(\rho, u)=2|\Pi|,
$$
where $u$ is a leaf of ${\mathcal T}_{\Sigma}$. But, by 
Corollary~\ref{lem:diam:bound},
$$\Delta({\mathcal T}_{\Sigma})\le 2|\Pi'|$$
for all partition systems $\Pi'\in\setpa(\Sigma)$. Thus $|\Pi|\le |\Pi'|$ 
for all partition systems $\Pi'\in \setpa(\Sigma)$ and so $\Pi$
is $\setpa({\mathcal T}_{\Sigma})$-minimum.

We prove the converse by establishing that if $\Pi$ is not hierarchical
then $\Pi$ is not $\setpa({\mathcal T}_{\Sigma})$-minimum.
Suppose that $\Pi$ is not hierarchical. Then there 
exist distinct $\pi_1, \pi_2\in \Pi$ with $A_1\in \pi_1$ and $A_2\in \pi_2$ 
such that $A_1\cap A_2\not\in \{\emptyset, A_1, A_2\}$. 
Let ${\mathcal T}^{\rho}_{\Sigma}$ denote the rooted weak $X$-tree 
obtained by viewing ${\mathcal T}_{\Sigma}$ rooted at $\rho$. Since 
$A_1\cap A_2\not\in \{\emptyset, A_1, A_2\}$, either $A_1$ or $A_2$ is not a 
cluster of ${\mathcal T}^{\rho}_{\Sigma}$. Without loss of generality, we may 
assume that $A_2$ is not a cluster of $\mathcal T^{\rho}_{\Sigma}$.
Let $E_{\pi_2}$ 
denote a subset of edges of $\mathcal T_{\Sigma}$ that displays $\pi_2$ and 
let $e$ denote the 
edge in $E_{\pi_2}$ displaying $A_2|(X-A_2)$. Observe that as $A_2$ is not 
a cluster of ${\mathcal T}^{\rho}_{\Sigma}$, it is easily seen that,
for each edge $e'\in E_{\pi_2}-\{e\}$, the 
unique path in $\mathcal T_{\Sigma}$ from $\rho$ to the vertex of 
$e'$ closer to $\rho$ traverses $e$. Now, by Theorem~\ref{thm:tree:sc}, 
${\mathcal T}_{\Sigma}$ is an even $X$-tree, and so, by Lemma~\ref{tinyeven}, 
${\mathcal T}_{\Sigma}/E_{\pi_2}$ is an even $X$-tree. We show next that
\begin{align}
\Delta({\mathcal T}_{\Sigma})=\Delta({\mathcal T}_{\Sigma}/E_{\pi_2}).
\label{same}
\end{align}

If the degree of $\rho$ is at least three, then, by the previous 
observation on the unique path in $\Tt_{\Sigma}$ starting at $\rho$, there
must exist leaves $x$ and $y$ in $\Tt_{\Sigma}$
such that the path from $\rho$ to either of them does not traverse
an edge in $E_{\pi_2}$. Thus,
$$
\Delta(\Tt_{\Sigma}) \geq \Delta({\mathcal T}_{\Sigma}/E_{\pi_2})
\geq d_{{\mathcal T}_{\Sigma}/E_{\pi_2}}(x,y)=d_{\mathcal T_{\Sigma}}(x,y)
=\Delta(\Tt_{\Sigma}),
$$
by Theorem~\ref{hequal}. Consequently, (\ref{same}) must hold in this case. 

So assume that the degree of $\rho$ equals two. Since, by assumption,
$\mathbb P(\Sigma)$ contains as hierarchical partition system $\Pi_h$
and 
$\Tt_{\Sigma}\cong \Tt_{\Pi_h}$, it follows by Corollary~\ref{equal}
that $\Tt_{\Sigma}$  does not contain an interior vertex that is labelled.
Since $A_2$ is not a cluster of $\Tt_{\Sigma}^{\rho}$, it follows that
$\Tt_{\Sigma}^{\rho}$ must contain a vertex of degree at least three
on the path from $\rho$ to 
the closer one of the two vertices of $e$. By the
observation above on the unique path in $\Tt_{\Sigma}$ starting at $\rho$
the same arguments as in the case that $\rho$ is of degree at least three
imply that (\ref{same}) must hold in this case too.


To complete the proof of the converse, let $\Pi_h$ denote again 
a hierarchical partition system in $\setpa(\Sigma)$. Then, 
combining Corollary~\ref{equal}, (\ref{same}), and 
Corollary~\ref{lem:diam:bound},
\begin{align*}
2|\Pi_h|=\Delta({\mathcal T}_{\Sigma})=
\Delta({\mathcal T}_{\Sigma}/E_{\pi_2})\le 2(|\Pi|-1) < 2|\Pi|.
\end{align*}
In particular, $|\Pi_h|<|\Pi|$, so $\Pi$ is not 
$\setpa({\mathcal T}_{\Sigma})$-minimum. 
This completes the proof of the converse and the theorem.
\end{proof}

\section{A Decision Problem}
\label{hard}

It could be of interest to try to extend the
main results in this paper to other types
of multisets of splits (e.\,g.\,weakly compatible 
or $k$-compatible sets \cite{GKM12}). 
For example, by Theorem~\ref{main1}, if
we are given a compatible multiset $\Sigma$ of
splits of a set $X$ it is easy to decide whether
or not there exists some partition system $\Pi$ on $X$ 
with $\Sigma_{\Pi}=\Sigma$, but what 
if $\Sigma$ is not compatible? We now prove
a result that indicates that extending our results
could be quite challenging.
In particular, we show that the following 
decision problem is NP-complete.

\noindent {\sc Partition System} \\ 
\noindent {\bf Instance}: A split system $\Sigma$ on $X$. \\
\noindent {\bf Question}: Is there a partition system $\Pi$ on $X$ 
such that $\Sigma_{\Pi}=\Sigma$?

To prove this result we first recall some useful facts.
Suppose $G$ is a graph. Then $G$ is called {\em simple}
if it does not contain a loop and {\em cubic} if every vertex
 has degree~$3$. A {\em matching} $M$ of $G$ is a subset of edges 
of $G$ such that no two edges in $M$ share a vertex.
A matching $M$ of $G$ is called {\em perfect}
if every vertex of $G$ is incident with some edge in $M$. 
A {\em $k$-edge colouring} of $G$ is an assignment of at 
most $k\geq 2$ colours to the edges of $G$ so that no two edges incident 
with the same vertex have the same colour. The {\em edge chromatic number} 
of $G$ is the smallest $k$ for which $G$ is $k$-edge colourable. 
A consequence of 
a theorem due to Vizing~\cite{viz64} is that the edge chromatic number 
of a simple cubic graph $G$ is either $3$ or $4$, where it is three if
and only if the edges of $G$ can be partitioned into three perfect 
matchings. 
To show that {\sc Partition System} is NP-complete, we use 
the following NP-complete problem~\cite{H81}:

\noindent {\sc Cubic Edge Colouring} \\
\noindent {\bf Instance:} A simple cubic graph $G$. \\
\noindent {\bf Question:} Is the edge chromatic number of $G$ three?

\begin{theorem}
\label{NP-complete}
The decision problem {\sc Partition System} is NP-complete
even if the split system $\Sigma$ is a \underline{set} of splits.
\end{theorem}

\begin{proof}
Clearly, {\sc Partition System} is in NP. Now, 
let $G$ be an instance of {\sc Cubic Edge Colouring} 
with vertex and edge sets $V$ and $E$, respectively. 
We may assume that $|V|\ge 5$. We construct an instance 
of {\sc Partition System} as follows. Let $X=V$ and let
\begin{align*}
\Sigma=\biguplus_{\{u,v\}\in E}\{ \{u,v\}|(X-\{u,v\})\}.
\end{align*}
Note that the time taken for this construction and the size 
of the constructed instance is polynomial in the size of $G$. Moreover
since $G$ is simple the multiplicity of each split in $\Sigma$ is one,
that is, $\Sigma$ is a set.
 We next show that there exists a partition system $\Pi$ on $X$ 
with $\Sigma_{\Pi}=\Sigma$ if and only if $E$ can be partitioned 
into three perfect matchings of $G$.

First suppose that $G$ has three pairwise-disjoint perfect 
matchings $M_1$, $M_2$, and $M_3$. 
Since $M_i$ is a partition of $X$ for all $i$ and since 
each edge $\{u,v\}$ of $G$ is in precisely one of $M_1$, 
$M_2$, and $M_3$, it follows that the partition system 
$\Pi'=\{M_1, M_2, M_3\}$ on $X$ has the property that $\Sigma_{\Pi'}=\Sigma$.

Now suppose that there is a partition system $\Pi$ on $X$ 
such that $\Sigma_{\Pi}=\Sigma$. Let $\pi\in \Pi$ and let 
$A\in \pi$. Since $A|(X-A)\in \Sigma$, either $A$ or $X-A$ 
is an edge of $G$. If $X-A$ is an edge of $G$, then, as 
$|X|\ge 5$, we have $\pi=\{A, X-A\}$. But then the multiplicity 
of $A|(X-A)$ in 
$\Sigma_{\Pi}$ is at least two; a contradiction as $\Sigma$ 
is a set and not a multiset. Therefore $A$ is an edge of $G$. 
As each vertex is incident with exactly $3$ edges, it now 
follows that $\Pi$ consists of three partitions of $X$ with 
each partition being a perfect matching of $G$. Since these 
matchings are pairwise disjoint, $E$ can be partitioned into 
three perfect matchings. This completes the proof of the theorem.
\end{proof}

\section{Discussion}
\label{sec:discussion}

In this section, we shall consider the mapping that takes a partition system $\Pi$ to the split system $\Sigma_\Pi$ in a more general setting. The study of similar mappings between combinatorial objects relevant to phylogenetic analysis, such as
split systems and distances, has proven to be
a fruitful approach to various problems in
the area of phylogenetic combinatorics (cf. e.g. \cite{DHMKS}).

We begin with some additional terminology and notation. 
Given a finite set $X$, let $\upi(X)$ and $\usigma(X)$ 
be the {\em set} of partitions and splits of $X$, respectively. 
In addition, for a subset $A$ of $X$, let $\upi(X;A)$ 
be the set of partitions $\pi$ in $\upi(X)$ with $A\in \pi$. 
A {\em real partition family} on $X$ is a map $\mu$ from $\upi(X)$ 
into $\R_{\geq 0}$, and a {\em real split family} on $X$ is a 
map $\nu$ from $\usigma(X)$ into $\R_{\geq 0}$. 
Moreover, $\mu$ is called an {\em integral partition family} if  $\mu(\pi)$ is a non-negative integer for every $\pi \in \upi(X)$, and {\em integral split families} are defined in a similar manner. Note that each partition system $\Pi$ on $X$ gives rise to a real partition family $\mu_{\Pi}$ on $X$ in which $\mu_{\Pi}$ maps each partition $\pi$ in $\upi(X)$ to the multiplicity of $\pi$ in $\Pi$ if $\pi\in \Pi$, and $0$ otherwise. 

Now consider the map
$
\kappa: \R_{\geq 0}^{\upi(X)} \longrightarrow \R_{\geq 0}^{\usigma(X)}
$
that takes a real partition family $\mu$ on $X$ to the real 
split family $\kappa(\mu)$ on $X$ defined by 
$$
\kappa(\mu)(A|B)=\sum_{\pi \in \upi(X;A)} \mu(\pi)+\sum_{\pi \in \upi(X;B)} \mu(\pi)
$$
for each split $A|B$ in $\usigma(X)$.  Then for a given partition system $\Pi$  on $X$ and each split $A|B$ in $\usigma(X)$, the value $\kappa(\mu_\Pi)(A|B)$ equals to the multiplicity of $A|B$ in the split system $\Sigma_{\Pi}$ if $A|B \in \Sigma_{\Pi}$, and $0$ otherwise. Therefore, the map $\kappa$ can be regarded as a generalization of the mapping that takes a partition system $\Pi$ on $X$ to the split system $\Sigma_\Pi$. 

In this framework, the results of the 
previous sections are mainly concerned with understanding 
the {\em kernel} of the map $\kappa$, that is, the 
set $\kappa^{-1}(\nu)$ for a real split family $\nu$ on $X$. In this context, we are especially interested in the case when $\nu$ is an integral split family and the {\em support} of $\nu$, defined as the set $\{S\in \usigma(X)\,:\,\nu(S)>0\}$, is compatible. In particular,  
Theorem~\ref{main1} presents a criterion to decide whether or not
the kernel $\kappa^{-1}(\nu)$ contains an integral partition family. 
Moreover, if such an integral partition family exists, then 
Theorem~\ref{thm:maximum} provides a canonical construction 
for a maximum-sized integral partition family in that kernel, and Algorithm \textsc{MinSizePartition} 
obtains an integral partition family in the kernel 
with the minimum size (see also Theorem~\ref{theo:verify}). 
Finally, as shown in Theorem~\ref{NP-complete}, if 
the support of $\nu$ is not compatible, then it 
is NP-complete to determine whether the kernel $\kappa^{-1}(\nu)$ 
contains an integral partition family.

\begin{figure}[h]
\center
\includegraphics[scale=0.3]{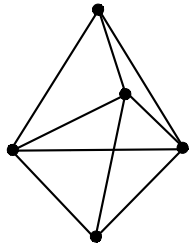}          
\caption{The 1-skeleton of the polytope $\kappa^{-1}(\nu_0)$ for the integral partition family $\nu_0$ given in Section~\ref{sec:discussion}.
}
\label{4-poly}
\end{figure}

In light of these facts, it would be interesting to characterize the 
set of real split families $\nu$ for which the kernel
$\kappa^{-1}(\nu)$ contains an integral partition family. Note that, given a real split family $\nu$, the kernel $\kappa^{-1}(\nu)$ 
may not contain an integral partition family, even if $\nu$ itself is integral. For example, consider the set $X=\{1,2,3,4\}$, the splits $S_i=\{\{i\},X-\{i\}\}$ for $1\leq i \leq 4$ and $S_5=\{\{1,2\}, \{3,4\}\}$, and let $\nu_0$ be the integral split family on $X$ defined by 
setting $\nu_0(S_i)=1$ for $1\leq i \leq 5$. 
Then it is straightforward to check that  $\kappa^{-1}(\nu_0)$ does not contain an integral partition family. However, it is not difficult to see that $\kappa^{-1}(\nu_0)$ is not empty and that it is in fact a three-dimensional polytope with five vertices (see Fig.~\ref{4-poly} for the  1-skeleton of $\kappa^{-1}(\nu_0)$ and~\cite{Z95} for definitions related to polytopes).  

More generally, it can be shown that the kernel $\kappa^{-1}(\nu)$ is always a polytope for each real split family $\nu$ on $X$. The proof of this fact is beyond the scope of this paper and will be presented elsewhere. 
Note that the polytope $\kappa^{-1}(\nu)$ can be much more complicated in general and there are several interesting questions that can be asked concerning its structure. For example, it could be of interest to find formulae for 
its dimension, and the number of its faces and vertices, or to find interesting 
characterizations for its faces and vertices. A better understanding of these questions should hopefully shed further light on mappings from partition systems to split systems and, ultimately, their application to phylogenetics.

\bigskip
\noindent
{\bf Acknowledgement} We would like to thank two anonymous referees and the editor L\'{a}szl\'{o} Sz\'{e}kely for their helpful comments, especially the suggestion to consider the mapping discussed in the final section.

\end{document}